\documentclass[10pt]{article} 
\date{27.10.2010}
\usepackage{amssymb}
\usepackage{amsmath}
\input{liemacs09.sty} % takes care of \qedhere 
\newcommand{\fB}{{\mathfrak B}}

\newcommand{\bL}{{\mathbb L}}

\renewcommand{\mlabel}{\label}

\begin{document} 

%%%%%%%%%%%%%%%%%%%%%%%%

%%%%%%%%%%%%%%%%%%%%%%%%%%

\title{Holomorphic Realization of Unitary Representations of 
Banach--Lie Groups} 

\author{Karl-Hermann Neeb\begin{footnote}{
Department  Mathematik, FAU Erlangen-N\"urnberg, Bismarckstrasse 1 1/2, 
91054-Erlangen, Germany; neeb@mi.uni-erlangen.de}
\end{footnote}
\begin{footnote}{Supported by DFG-grant NE 413/7-1, Schwerpunktprogramm 
``Darstellungstheorie''.} 
\end{footnote}}

\maketitle
\centerline{\sl To J. Wolf on the occasion of his 75th birthday} 

%{\bf This is {\tt herm4.tex}} 

\begin{abstract} In this paper we explore the method of 
holomorphic induction for unitary representations of 
Banach--Lie groups. First we show that the classification 
of complex bundle structures on homogeneous Banach bundles over 
complex homogeneous spaces of real Banach--Lie groups 
formally looks as in the finite dimensional case. 
We then turn to a suitable concept of 
holomorphic unitary induction and show that 
this process preserves commutants. In particular, holomorphic 
induction from irreducible representations leads to irreducible ones. 
Finally we develop criteria to identify representations 
as holomorphically induced ones and apply these to 
the class of so-called positive energy representations. 
All this is based on extensions of Arveson's concept of spectral subspaces to 
representations on Fr\'echet spaces, in particular on spaces of 
smooth vectors. 
\end{abstract} 

\section*{Introduction} 

This paper is part of a long term project concerned with a 
systematic approach to unitary representations 
of Banach--Lie groups in terms of conditions 
on spectra in the derived representation. A unitary 
representation $\pi \: G \to \U(\cH)$ 
is said to be {\it smooth} if the subspace 
$\cH^\infty$ of smooth vectors is dense. 
This is automatic for continuous 
representations of finite dimensional groups but not in general 
(cf.\ \cite{Ne10a}). For any smooth 
unitary representation, the {\it derived representation} 
\[ \dd\pi \: \g = \L(G)\to \End(\cH^\infty), \quad 
\dd\pi(x)v := \derat0 \pi(\exp tx)v\]  
carries significant information in the sense that the closure of the 
operator $\dd\pi(x)$ coincides with the infinitesimal generator of the 
unitary one-parameter group $\pi(\exp tx)$. We call $(\pi, \cH)$ 
{\it semibounded} if the function 
\[ s_\pi \: \g \to \R \cup \{ \infty\},\ \ 
s_\pi(x) 
:= \sup\big(\Spec(i\dd\pi(x))\big) 
= \sup\{ \la i\dd\pi(x)v,v\ra \: v \in \cH^\infty, \|v\|=1\}\] 
is bounded on the neighborhood of some point in $\g$ 
(cf.\ \cite[Lemma~5.7]{Ne08}). 
Then the set $W_\pi$ of all such points 
is an open invariant convex cone in the Lie algebra $\g$. 
We call $\pi$ {\it bounded} if $s_\pi$ is bounded on some $0$-neighborhood, 
which is equivalent to $\pi$ being norm continuous 
(cf.\ Proposition~\ref{prop:contin}). 
All finite dimensional unitary representations are bounded 
and many of the unitary representations appearing in physics are semibounded 
(cf. \cite{Ne10c}). 

One of our goals is a classification of the irreducible semibounded 
representations and the development of 
tools to obtain direct integral decompositions 
of semibounded representations. To this end, realizations 
of unitary representations in spaces of holomorphic sections of 
vector bundles turn out to be extremely helpful. 
Clearly, the bundles in question should be allowed to have 
fibers which are infinite dimensional Hilbert spaces, and to 
treat infinite dimensional groups, we also have to admit 
infinite dimensional base manifolds. 
The main point of the present paper is to provide 
effective methods to treat unitary representations 
of Banach--Lie groups in spaces of holomorphic sections of homogeneous 
Hilbert bundles. 

In Section~\ref{sec:1} we explain how to parametrize 
the holomorphic structures on Banach vector bundles 
$\bV = G \times_H V$ over a Banach homogeneous space $M = G/H$ 
associated to a norm continuous representation 
$(\rho, V)$ of the isotropy group~$H$. The main result 
of Section~\ref{sec:1} is 
Theorem~\ref{thm:a.2} which generalizes the corresponding 
results for the finite dimensional case by Tirao and Wolf 
(\cite{TW70}). As in  finite dimensions, the complex bundle 
structures are specified by  
``extensions'' $\beta \: \fq \to \gl(V)$ of the differential 
$\dd \rho \: \fh \to \gl(V)$ to a representation of the complex subalgebra 
$\fq \subeq \g_\C$ specifying the complex structure on~$M$ in the sense that 
$T_H(M) \cong \g_\C/\fq$ 
(cf.~\cite{Bel05}). The main point in \cite{TW70} is that the homogeneous space 
$G/H$ need not be realized as an open $G$-orbit in a complex 
homogeneous space of a complexification $G_\C$, which is impossible 
if the subgroup of $G_\C$ generated by $\exp \fq$ is not closed. 
In the Banach context, two additional difficulties appear: 
The Lie algebra $\g_\C$ may not be integrable in the sense that 
it does not belong to any Banach--Lie group (\cite{GN03}) and, even if $G_\C$ 
exists and the subgroup $Q := \la \exp \fq \ra$ is closed, it 
need not be a Lie subgroup so that there is no natural construction 
of a manifold structure on the quotient space~$G/Q$. 
As a consequence, the strategy of the proof 
in \cite{TW70} can not be used for Banach Lie groups. Another 
difficulty of the infinite dimensional context is that there is 
no general existence theory for solutions of $\oline\partial$-equations 
(see in particular \cite{Le99}). 

The next step, carried out in Section~\ref{sec:2},  
is to analyze Hilbert subspaces of the space $\Gamma(\bV)$ 
of holomorphic sections of $\bV$ 
on which $G$ acts unitarily. In this context 
$(\rho, V)$ is a bounded unitary representation. The most regular 
Hilbert spaces with this property are those that we call 
{\it holomorphically induced from $(\rho, \beta)$}. 
They contain $(\rho, V)$ as an $H$-subrepresentation 
satisfying a compatibility condition with respect to $\beta$. 
Here we show that if 
the subspace $V \subeq \cH$ is invariant under the commutant 
$B_G(\cH) = \pi(G)'$, then restriction to $V$ yields an isomorphism 
of the von Neumann algebra $B_G(\cH)$ with a suitably defined commutant 
$B_{H,\fq}(V)$ of $(\rho,\beta)$ (Theorem~\ref{thm:5.5}). 
This has remarkable 
consequences. One is that the representation of $G$ on $\cH$ 
is irreducible (multiplicity free, discrete, type I) if and only 
if the representation of $(H,\fq)$ on $V$ has this property. 
The second main result in Section~\ref{sec:2} is a 
criterion for a unitary representation 
$(\pi, \cH)$ of $G$ to be holomorphically induced 
(Theorem~\ref{thm:a.3}). 

Section~\ref{sec:3} is devoted to a description 
of environments in which Theorem~\ref{thm:a.3} applies naturally. 
Here we consider an element $d \in \g$ which is {\it elliptic} in 
the sense that the one-parameter group $e^{\R \ad d}$ of automorphisms 
of $\g$ is bounded. This is equivalent to the existence of an 
invariant compatible norm. Suppose that $0$ is isolated in $\Spec(\ad d)$. 
Then the subgroup $H := Z_G(d) = \{ g \in G \: \Ad(g)d = d \}$ 
is a Lie subgroup and the homogeneous space $G/H$ carries a natural 
complex manifold structure. A smooth unitary representation 
$(\pi, \cH)$ of $G$ is said to be of 
{\it positive energy} if the selfadjoint operator 
$-i \dd\pi(d)$ is bounded from below. Note that this is in particular 
the case if $\pi$ is semibounded with $d \in W_\pi$. 
Any positive energy representation 
is generated as a $G$-representation by the closed subspace 
$V := \oline{(\cH^\infty)^{\fp^-}}$, where 
$\fq = \fp^+ \rtimes \fh_\C \subeq \g_\C$ is the complex subalgebra 
defining the complex structure on $G/H$ and 
$\fp^- := \oline{\fp^+}$. 
If the $H$-representation $(\rho, V)$ is bounded, then 
$(\pi, \cH)$ is holomorphically induced 
by $(\rho,\beta)$, where $\beta$ is determined by  $\beta(\fp^+) = \{0\}$ 
(Theorem~\ref{thm:6.2}). In Theorem~\ref{thm:3.15} we further 
show that, under these assumptions, $\pi$ is semibounded with $d \in W_\pi$. 
These results are rounded off by Theorem~\ref{thm:6.2b} which shows 
that, if $\pi$ is semibounded with $d \in W_\pi$, then 
$\pi$ is a direct sum of holomorphically induced representations 
and if, in addition, $\pi$ is irreducible, then $V$ coincides 
with the minimal eigenspace of $-i\dd\pi(d)$ and the representation 
of $H$ on this space is automatically bounded. 
For all this we use refined analytic tools based on the fact that 
the space 
$\cH^\infty$ of smooth vectors is a Fr\'echet space on which 
$G$ acts smoothly (\cite{Ne10a}) and the $\R$-action on 
$\cH^\infty$ defined by $\pi_d(t) := \pi(\exp td)$ is equicontinuous. 
These properties permit us to use a 
suitable generalization of Arveson's spectral theory, developed
 in Appendix~\ref{app:1}. 

In \cite{Ne11} we shall use the techniques developed 
in the present paper to obtain a complete descriptions of semibounded 
representations for the class of hermitian Lie groups. 
One would certainly like to extend the tools developed here to 
Fr\'echet--Lie groups such as diffeomorphism groups and groups of 
smooth maps. Here a serious problem is the construction of 
holomorphic vector bundle structures on associated bundles, 
and we do not know how to extend this beyond the Banach context, 
especially because of the non-existing solution theory for 
$\oline\partial$-equations (cf.\ \cite{Le99}). 
%However, we hope that the context of covariant representations 
%developed in Section~\ref{sec:3} can also be used to 
%study representations of Fr\'echet--Lie groups such that 
%$G^\infty \rtimes_\alpha \R$, where $G^\infty \subeq G$ is the 
%Fr\'echet--Lie group of those elements of $G$ whose $\alpha$-orbit 
%map is smooth. 

Several of our results have natural predecessors in more restricted
 contexts. In \cite{BR07} Belti\c{t}\u{a} and Ratiu 
study holomorphic Hilbert bundles over Banach manifolds $M$ 
and consider the endomorphism bundle $B(\bV)$  over $M \times \oline M$ 
(using a different terminology). They also relate Hilbert subspaces 
$\cH$ of $\Gamma(\bV)$ to 
reproducing kernels, which in this context are 
sections of $B(\bV)$ satisfying a certain holomorphy condition 
which under the assumption of local boundedness of the kernel 
is equivalent to 
holomorphy as a section of $B(\bV)$ (\cite[Thm.~4.2]{BR07}; 
see also \cite[Thm.~1.4]{BH98} and \cite{MPW97} for the case of 
finite dimensional 
bundles over finite dimensional manifolds and \cite{Od92} for the 
case of line bundles). 
Belti\c{t}\u{a} and Ratiu 
use this setup to realize certain representations  
of a $C^*$-algebra $\cA$, which define bounded unitary representations 
of the unitary group $G := \U(\cA)$, in spaces of holomorphic sections of a 
bundle over a homogeneous space of the unit group $\cA^\times$ on 
which $\U(\cA)$ acts transitively (\cite[Thm.~5.4]{BR07}). 
These homogeneous spaces are of the form 
$G/H$, where $H$ is the centralizer of some hermitian element 
$a \in \cA$ with finite spectrum, so that the realization 
of these representations by holomorphic sections could also be derived 
from our Corollary~\ref{cor:6.2}. This work has been continued by 
Belti\c{t}\u{a} with Gal\'e in a different direction, focusing 
on complexifications of real homogeneous spaces 
instead of invariant complex structures on the real spaces (\cite{BG08}). 

In \cite{Bo80} Boyer constructs irreducible unitary representations 
of the Hilbert--Lie group 
$\U_2(\cH) = \U(\cH) \cap (\1 + B_2(\cH))$ via holomorphic induction 
from characters of the diagonal subgroup. They live in spaces of 
holomorphic sections on homogeneous spaces of the complexified 
group $\GL_2(\cH) = \GL(\cH) \cap (\1 + B_2(\cH))$, 
which are restricted versions of flag manifolds carrying 
strong K\"ahler structures.

In the present paper we only deal with representations 
associated to homogeneous vector bundles. To understand branching laws 
for restrictions of representations to subgroups, one should 
also study situations where the group $G$ does not act transitively 
on the base manifold. For finite dimensional holomorphic vector 
bundles, this has been done extensively by 
T.~Kobayashi who obtained powerful criteria for representations 
in Hilbert spaces of holomorphic sections to be multiplicity 
free (cf.\ \cite{KoT05}, \cite{KoT06}). It would be interesting to explore 
the extent to which 
Kobayashi's technique of visible actions on complex manifolds 
can be extended to Banach manifolds. \\

{\bf Notation:} For a group $G$ we write $\1$ for the neutral element 
and $\lambda_g(x) = gx$, resp., $\rho_g(x) = xg$ 
for left multiplications, resp., right multiplications. 
We write $\g_\C$ for the complexification of a real Lie algebra $\g$ and 
$\oline{x + iy} := x - iy$ for the complex conjugation on $\g_\C$, which is an 
antilinear Lie algebra automorphism. 
For two Hilbert spaces $\cH_1, \cH_2$, we write 
$B(\cH_1, \cH_2)$ for the space of continuous (=bounded) linear operators 
$\cH_1 \to \cH_2$ and $B_p(\cH)$, $1 \leq p < \infty$, 
for the space of Schatten class operators 
$A \: \cH \to \cH$ of order $p$, i.e., $A$ is compact with 
$\tr((A^*A)^{p/2}) < \infty$. \\

{\bf Acknowledgment:} We thank Daniel Belti\c{t}\u{a} for a careful 
reading of earlier versions of this paper, for pointing out references 
and for various interesting discussions on its subject matter. 

\tableofcontents

\section{Holomorphic Banach bundles} \mlabel{sec:1}

To realize unitary representations 
of Banach--Lie groups in spaces of holomorphic sections of 
Hilbert bundles, we first need a 
parametrization of holomorphic bundle structures on 
given homogeneous vector bundles for real Banach--Lie groups. 
As we shall see in Theorem~\ref{thm:a.2} below, 
formulated appropriately, the corresponding results
from the finite dimensional case (cf.\ \cite{TW70}) 
can be generalized to the Banach context. 

The following observation provides some information 
on the assumptions required on the isotropy representation 
of a Hilbert bundle. It is a slight modification 
of results from \cite[Sect.~3]{Ne09}. 

\begin{prop}\label{prop:contin} Let $(\pi,\cH)$ be a 
unitary representation of the Banach--Lie group $G$ 
for which all vectors are smooth. 
Then $\pi \: G \to \U(\cH)$ is a morphism of 
Lie groups, hence in particular norm continuous, i.e., a 
bounded representation. 
\end{prop}

\begin{prf} Our assumption implies that, for each $x\in \g$, 
the infinitesimal generator $\dd\pi(x)$ of the unitary 
one-parameter group $\pi_x(t) := \pi(\exp_G(tx))$ 
is everywhere defined, hence a bounded operator because its 
graph is closed. 

Therefore the derived representation leads to a morphism of 
Banach--Lie algebras 
$\dd\pi \: \g \to B(\cH).$
Since the function $s_\pi$ 
is a sup of a set of continuous linear functionals, it is 
lower semi-continuous. Hence the function 
\[ x \mapsto \|\dd\pi(x)\| 
= \max(s_\pi(x), s_\pi(-x)) \] 
is a lower semi-continuous seminorm 
 and therefore continuous because $\g$ is barreled 
(cf.\ \cite[\S III.4.1]{Bou07}).
We conclude that the set of all linear functional 
$\la \dd\pi(\cdot)v,v\ra$, $v \in \cH^\infty$ a unit vector, 
 is equicontinuous in  $\g'$, 
so that the assertion follows from \cite[Thm.~3.1]{Ne09}. 
\end{prf}

We now turn to the case where $M$ is a Banach homogeneous space. 
Let $G$ be a Banach--Lie group with Lie algebra $\g$ and 
$H \subeq G$ be a split Lie subgroup, i.e., 
the Lie algebra $\fh$ of $H$ has a closed complement in $\g$, 
 for which the coset space 
$M := G/H$ carries the structure of a complex manifold such that 
the projection $q_M\: G\to G/H$ is a smooth $H$-principal bundle and 
$G$ acts on $M$ by holomorphic maps. 
Let $m_0 = q_M(\1) \in M$ be the canonical base point and 
$\fq \subeq \g_\C$ be the kernel of the complex linear extension 
of the map $\g \to T_{m_0}(G/H)$ to $\g_\C$, so that $\fq$ is a closed 
subalgebra of~$\g_\C$ invariant under $\Ad(H)$ 
(cf.\ \cite[Thm.~15]{Bel05}). We call $\fq$ the 
{\it subalgebra defining the complex structure on $M = G/H$} 
because specifying $\fq$ means to identify 
$T_{m_0}(G/H)\cong \g/\fh$ with the complex Banach space $\g_\C/\fq$ 
and thus specifying the complex structure on~$M$.

\begin{rem} \mlabel{rem:norm-cont} 
If the Banach--Lie group $G$ acts smoothly by isometric bundle automorphisms 
on the holomorphic Hilbert bundle $\bV$ over $M = G/H$, then 
the action of the stabilizer group $H$ on $V := \bV_{m_0}$ 
is smooth, so that Proposition~\ref{prop:contin} shows that it 
defines a bounded unitary representation 
$\rho \: H \to \U(V)$. 

If $\sigma_M \: G \times M \to M$ denotes the corresponding action on 
$M$ and $\dot\sigma_M \: \g \to \cV(M)$ the derived action, 
then, for the closed subalgebra 
\[ \fq := \{ x \in \g_\C \: \dot\sigma_M(x)(m_0) = 0\}, \] 
the representation $\beta \: \fq \to B(V)$ 
is given by a continuous bilinear map 
\[ \hat\beta \: \fq \times V \to V, \quad 
(x,v) \mapsto \beta(x)v.\] 
This means that $\beta$ is a continuous morphism 
of Banach--Lie algebras. 
\end{rem}

This observation leads us to the following structures.

\begin{defn}\mlabel{def:a.1} Let $H \subeq G$ be a 
Lie subgroup and  $\fq \subeq \g_\C$ be a closed subalgebra containing 
$\fh_\C$. 
If  $\rho \: H \to \GL(V)$ is a norm continuous representation 
on the Banach space $V$, then a morphism 
$\beta \: \fq \to \gl(V)$ of complex Banach--Lie algebras 
is said to be an {\it extension of $\rho$} if 
\begin{equation}\label{eq:comprel} 
\dd\rho = \beta\res_\fh \quad \mbox{ and } \quad 
\beta(\Ad(h)x) = \rho(h)\beta(x)\rho(h)^{-1} \quad \mbox{ for } \quad 
h \in H, x \in \fq. 
\end{equation}
\end{defn}

\begin{defn} (a) 
If $q \: \bV = G \times_H V \to M$ is a homogeneous vector bundle 
defined by the norm continuous representation $\rho \: H \to \GL(V)$, 
we associate to each section $s \: M \to \bV$ the function 
$\hat s \: G \to V$ specified by $s(gH) = [g, \hat s(g)]$. 
A function $f \: G \to V$ is of the form $\hat s$ for a 
section of $\bV$ if and only if 
\begin{equation} \label{eq:equiv-sec} 
f(gh) = \rho(h)^{-1} f(g) \quad \mbox{ for } \quad 
g \in G, h \in H. 
\end{equation}
We write $C(G,V)_\rho$, resp., $C^\infty(G,V)_\rho$ for the 
continuous, resp., smooth functions satisfying 
\eqref{eq:equiv-sec}.  

(b) We associate to each 
$x \in \g_\C$ the left invariant 
differential operator on $C^\infty(G,V)$ defined by 
\[ (L_x f)(g) := \derat0 f(g\exp(tx)) \quad \mbox{ for } \quad 
x \in \g. \]
By complex linear extension, we define the operators 
\[L_{x+ iy} := L_x + i L_y \quad \mbox{ for } \quad z = x + iy \in \g_\C, 
x,y \in \g. \]

(c) For any extension $\beta$ of $\rho$, 
we write $C^\infty(G,V)_{\rho,\beta}$ for the 
subspace of those elements $f \in C^\infty(G,V)_\rho$ 
satisfying, in addition,  
\begin{equation} \label{eq:inf-equiv}
L_x f = - \beta(x) f \quad \mbox{ for } \quad x \in \fq. 
\end{equation}
\end{defn}

\begin{rem}
For $x \in \g$, $h \in H$ and any smooth function 
$\phi$ defined on an open right $H$-invariant subset of 
$G$, we have 
\begin{eqnarray*}
(L_x\phi)(gh) 
&=& \derat0 \phi\big(g\exp(t\Ad(h)x)h\big)=
 (L_{\Ad(h)x}(\phi \circ \rho_h))(g), 
\end{eqnarray*}
so that we obtain for each $x \in \g_\C$ and $h \in H$ the 
relation 
\begin{eqnarray}\label{eq:liederrel}
(L_x\phi)\circ \rho_h = L_{\Ad(h)x}(\phi \circ \rho_h). 
\end{eqnarray}
\end{rem}

The proof of the following theorem is very much inspired by \cite{TW70}.
%\begin{footnote}{Here we should also insert a reference 
%to Martin Laubinger's new paper on solutions of the $\oline\partial$-
%Maurer--Cartan equation in the Banach context.}  
%\end{footnote}

\begin{thm} \mlabel{thm:a.2} Let $M = G/H$, $V$ be a complex Banach space  
and $\rho \: H \to \GL(V)$ be a norm continuous representation.  
Then, for any extension $\beta \: \fq \to \gl(V)$ of $\rho$, 
the associated bundle $\bV := G \times_H V$ carries 
a unique structure of a holomorphic vector bundle over $M$, 
which is determined by the 
characterization of the holomorphic sections $s \: M \to \bV$ 
as those for which $\hat s \in C^\infty(G,V)_{\rho,\beta}$. 
Any such holomorphic bundle structure is $G$-invariant in the sense that 
$G$ acts on $\bV$ by holomorphic bundle automorphism. 
Conversely, every $G$-invariant holomorphic vector bundle 
structure on $\bV$ is obtained from this construction.
\end{thm}

\begin{prf} {\bf Step 1:} Let $\beta$ be an extension 
of $\rho$ and  $E \subeq \g$ be a closed subspace 
complementing $\fh$. Then $E \cong \g/\fh \cong T_{m_0}(M)$ 
implies the existence 
of a complex structure $I_E$ on $E$ for which $E \to \g/\fh$ is an isomorphism 
of complex Banach spaces. Therefore 
the $I_E$-eigenspace decomposition 
\[ E_\C = E_+ \oplus E_-, \quad E_\pm = \ker(I_E \mp i\1),\] 
is a direct decomposition into closed subspaces. The 
quotient map $\g_\C \to \g_\C/\fq \cong \g/\fh$ 
is surjective on $E_+$ and annihilates $E_-$. Now 
$\fr := E_+$ is a closed complex complement of $\fq 
= \fh_\C \oplus E_-$ in $\g_\C$ (\cite[Thm.~15]{Bel05}). 

Pick open convex $0$-neighborhoods 
$U_\fr \subeq \fr$ and $U_\fq \subeq \fq$ such that the BCH multiplication 
$*$ defines a biholomorphic map 
$\mu \: U_\fr \times U_\fq \to U, (x,y) \mapsto x * y$, 
onto the open $0$-neighborhood $U \subeq \g_\C$ 
and that $*$ defines an associative multiplication 
defined on all triples of elements of $U$. 

{\bf Step 2:} On the $0$-neighborhood  $U \subeq \g_\C$, we consider the 
holomorphic function 
$$ F \: U \to \GL(V), \quad 
F(x * y) := e^{-\beta(y)}. $$
Let $U_\g \subeq U$ be an open $0$-neighborhood which is 
mapped by $\exp_G$ diffeomorphically onto an open $\1$-neighborhood 
$U_G$ of $G$. Then we consider the smooth function 
$$ f \: U_G \to \GL(V), \quad \exp_G z \mapsto F(z). $$ 

For $w \in \g$ and $z \in U_\g$, the BCH product $z * tw \in \g_\C$ 
is defined if 
$t$ is small enough, and we have 
\begin{equation}\label{eq:1.4} 
(L_w f)(\exp_G z) = \derat0 F(z * tw) 
= \dd F(z) \dd\lambda_z^*(0)w, 
\end{equation}
where $\dd\lambda_z^*(0) \: \g_\C \to \g_\C$ is the differential 
of the multiplication map $\lambda_z^*(x) = z * x$ in $0$.
As \eqref{eq:1.4} is complex linear in $w \in\g_\C$, it follows that 
$$ (L_w f)(\exp_G z) = \dd F(z) \dd\lambda_z^*(0)w $$  
hold for every $w \in \g_\C$, hence in particular for $w \in \fq$. 
For $w \in \fq$ and $z = \mu(x,y)$, we thus obtain 
\begin{align*}
\dd F(z) \dd\lambda_z^*(0)w 
&= \derat0 F(x * y * tw)
= \derat0 e^{-\beta(y * tw)} 
= \derat0  e^{-t \beta(w)} e^{-\beta(y)}\\
&=  -\beta(w) e^{-\beta(y)}. 
\end{align*}
We conclude that 
$$ (L_w f)(g) = -\beta(w) f(g) \quad \mbox{ for } \quad g \in U_G. $$ 
In particular, we obtain 
$f(gh) = \rho(h)^{-1} f(g)$ for $g \in U_G, h \in H_0.$ 

{\bf Step 3:} 
Since $H$ is a complemented Lie subgroup, there 
exists a connected submanifold $Z \subeq G$ containing $\1$ for which 
the multiplication map 
$Z\times H \to G, (x,h) \mapsto xh$ is a diffeomorphism onto 
an open subset of $G$. Shrinking $U_G$, we may therefore 
assume that $U_G = U_Z U_H$ holds for a connected open $\1$-neighborhood 
$U_Z$ in $Z$ and a connected open $\1$-neighborhood 
$U_H$ in $H$. Then 
\[ \tilde f(z h) := \rho(h)^{-1} f(z) 
\quad \mbox{ for } \quad z \in U_Z, h \in H, \] 
defines a smooth function $\tilde f \: U_Z H \to \GL(V)$. 
That it extends $f$ follows from the fact that 
$u = zh \in U_G$ with $z \in U_Z$ and $h \in U_H$ implies 
$h \in H_0$, so that 
$$ \tilde f(u) 
= \rho(h)^{-1} f(z)
= f(zh) = f(u). $$

For $w \in \fq$, formula \eqref{eq:liederrel}
leads to 
\begin{align*}
(L_w \tilde f)(zh) 
&= (L_{\Ad(h)w}(\tilde f \circ \rho_h))(z)
= L_{\Ad(h)w}(\rho(h)^{-1}\tilde f)(z)
= L_{\Ad(h)w}(\rho(h)^{-1}f)(z)\\
&= - \rho(h)^{-1}\beta(\Ad(h)w)f(z) 
= - \beta(w)\rho(h)^{-1}f(z) 
= - \beta(w)\tilde f(zh) . 
\end{align*}
Therefore $\tilde f$ satisfies 
\begin{equation}\label{eq:comprel2}
L_w \tilde f = - \beta(w) \tilde f \quad \mbox{ for } \quad w \in \fq.
\end{equation}

{\bf Step 4:} For $m \in M$ we choose an element $g_m \in G$ with $g_m.m_0 = m$ 
and put $U_m := g q_M(U_Z)$, so that 
$$ G^{U_m} := q_M^{-1}(U_m) = g_m U_Z H. $$ 
On this open subset of $G$, the function 
$$ F_m \: G^{U_m} \to \GL(V), \quad 
F_m(g) := \tilde f(g_m^{-1}g) $$ 
satisfies 
\begin{description}
\item[\rm(a)] $F_m(gh) = \rho(h)^{-1} F_m(g)$ for $g \in G^{U_m}$, $h \in H$.
\item[\rm(b)] $L_w F_m = -\beta(w) F_m$ for $w \in \fq$. 
\item[\rm(c)] $F_m(g_m) = \1 = \id_V$. 
\end{description}

Next we note that the function 
$F_m$ defines a smooth trivialization 
$$ \phi_m \: U_m \times V \to \bV\res_{U_m} = G^{U^m} \times_H V, \quad 
(gH, v) \mapsto [g, F_m(g)v]. $$
The corresponding transition functions are given by 
$$ \phi_{m,n} \: U_{m,n} := U_m \cap U_n \to \GL(V), \quad 
gH \mapsto \tilde\phi_{m,n}(g) := F_m(g)^{-1} F_n(g). $$ 

To verify that these transition functions are holomorphic, 
we have to show that the functions 
$\tilde\phi_{m,n} \: G^{U_{m,n}} \to B(V)$ 
are annihilated by the differential operators $L_w$, $w \in \fq$. 
This follows easily from the product rule and (b): 
\begin{align*}
L_w(F_m^{-1} F_n) 
&= L_w(F_m^{-1}) F_n + F_m^{-1} L_w(F_n)
= - F_m^{-1} L_w(F_m) F_m^{-1} F_n + F_m^{-1} L_w(F_n)\\
&= F_m^{-1}(\beta(w) F_m F_m^{-1}- \beta(w)) F_n 
= F_m^{-1}(\beta(w) - \beta(w)) F_n = 0.
\end{align*}
We conclude that the transition functions $(\phi_{m,n})_{m,n \in M}$ 
define a holomorphic vector bundle atlas on $\bV$. 

{\bf Step 5:} To see 
that this holomorphic structure is determined uniquely 
by $\beta$, pick an open connected subset $U \subeq M$ containing 
$m_0$ on which the bundle is holomorphically trivial and 
the trivialization is specified by a 
smooth function $F \: G^U \to \GL(V)$. 
Then $L_w F = -\beta(w)F$ implies that 
$\beta(w) = - (L_wF)(\1) F(\1)^{-1}$ 
is determined uniquely by the  holomorphic structure on $\bV$. 

{\bf Step 6:} The above construction also shows that, for 
any holomorphic vector bundle structure on $\bV$, for which 
$G$ acts by holomorphic bundle automorphisms, we may consider 
\begin{equation}
  \label{eq:beta}
\beta \: \fq \to B(V), \quad \beta(w) := - (L_wF)(\1) F(\1)^{-1} 
\end{equation}
for a local trivialization given by $F \: G^U \to \GL(V)$, where 
$m_0 \in U$. Then 
$\beta$ is a continuous linear map. 
To see that it does not depend on the choice of $F$, we note that, 
for any other trivialization $\tilde F \: G^U \to \GL(V)$, 
the function $F^{-1}\cdot \tilde F$ factors through a holomorphic 
function on $U$, so that 
$$ 0 = L_w(F^{-1}\tilde F)(\1) 
= F(\1)^{-1}\Big(- (L_w F)(\1)F(\1)^{-1} + (L_w \tilde F)(\1)\tilde F(\1)^{-1}\Big)\tilde F(\1). $$
We conclude that the right hand side of \eqref{eq:beta} 
does not depend on the choice of $F$. 
Applying this to functions of the form 
$F_g(g') := F(gg')$, defined on $g^{-1}G^U$, we obtain in particular 
\[ -\beta(w) = -(L_w f)(\1)F(\1)^{-1} = (L_w F_g)(\1) F_g(\1)^{-1}
= (L_w F)(g) F(g)^{-1}, \]
so that $L_w F = -\beta(w)F$. 

For $g= h \in H$, we obtain  with \eqref{eq:liederrel} 
\begin{align*}
-\beta(w) &= (L_w F_h)(\1) F(h)^{-1}
= (L_w F)(h) F(\1)^{-1}\rho(h)\\
&= \rho(h)^{-1}(L_{\Ad(h)w}F)(\1) F(\1)^{-1}\rho(h)
= -\rho(h)^{-1}\beta(\Ad(h)w)\rho(h).
\end{align*}
For $w \in \fh$, the relation $\beta(w) = \dd\rho(w)$ 
follows immediately from the $H$-equivariance of~$F$. 
To see that $\beta$ is an extension of $\rho$, it now remains 
to verify that it is a homomorphism of Lie algebras. 
For $w_1, w_2 \in \fq$, we have 
\begin{align*}
\beta([w_1, w_2])F 
&= - L_{[w_1, w_2]}F
= L_{w_2} L_{w_1} F - L_{w_1} L_{w_2} F \\
&= -L_{w_2} \beta(w_1) F + L_{w_1} \beta(w_2)F 
= \beta(w_1)  \beta(w_2)F - \beta(w_2)\beta(w_1)F,
\end{align*}
which shows that $\beta$ is a homomorphism of Lie algebras. 
\end{prf} 

\begin{ex}  We consider the special case  
where $G$ is contained as a Lie subgroup in 
a complex Banach--Lie group $G_\C$ with Lie algebra $\g_\C$ 
and where $Q := \la \exp \fq \ra$ is a Lie subgroup of $G_\C$ 
with $Q \cap G = H$. Then the orbit mapping $G \to G_{\C}/Q, 
g \mapsto g Q$, induces an open 
embedding of $M = G/H$ as an open $G$-orbit in the complex manifold
$G_{\C}/Q$.

In this case every holomorphic representation 
$\pi \: Q \to \GL(V)$ defines an associated holomorphic Banach 
bundle $\bV := G_\C \times_Q V$ over the complex manifold 
$G_\C/Q \cong G/H$. 

Since the Lie algebra $\g_\C$ need not be integrable in the sense 
that it is the Lie algebra of a Banach--Lie group 
(cf.~\cite[Sect.~6]{GN03}), 
the assumption $G \subeq G_\C$ is not general enough to 
cover every situation. One therefore needs the general Theorem~\ref{thm:a.2}.

In \cite[Sect.~6]{GN03} one finds examples 
of simply connected Banach--Lie groups 
$G$ for which $\g_\C$ is not integrable. Here 
$G$ is a quotient of a simply connected Lie group 
$\hat G$ by a central subgroup $Z \cong \R$, 
$\hat G$ has a simply connected universal complexification 
$\eta_{\hat G} \: {\hat G} \to {\hat G}_\C$, and the subgroup 
$\exp_{{\hat G}_\C}(\fz_\C)$ is not closed; its closure is a $2$-dimensional 
torus. 

Then the real Banach--Lie group $G$ acts smoothly and faithfully 
by holomorphic maps on a complex Banach 
manifold $M$ for which $\g_\C$ is not integrable. 
It suffices to pick a suitable tubular neighborhood 
${\hat G} \times U \cong {\hat G} \exp(U) \subeq {\hat G}_\C$, 
where 
$U \subeq i\hat\fg$ is a convex $0$-neighborhood. 
Then the quotient 
$M := {\hat G}/Z \times (U + i\z)/i\z$ is a complex manifold on which 
$G \cong {\hat G}/Z$ acts faithfully, but $\g_\C$ is not integrable. 
\end{ex}

\section{Hilbert spaces of holomorphic sections} \mlabel{sec:2}

In this section we take a closer look at 
Hilbert spaces of holomorphic sections in $\Gamma(\bV)$ for 
holomorphic Hilbert bundles $\bV$ constructed with the methods from 
Section~\ref{sec:1}. Here we are interested in Hilbert spaces with 
continuous point evaluations (to be defined below) 
on which $G$ acts unitarily. This leads to the concept of a 
holomorphically induced representation, which for infinite dimensional 
fibers is a little more subtle than in the finite dimensional case. 
The first key result of this section is 
Theorem~\ref{thm:5.5} relating the commutant of a 
holomorphically induced unitary $G$-representation 
$(\pi, \cH)$ to the commutant of the representation 
$(\rho,\beta)$ of $(H,\fq)$ on the fiber~$V$.
This result is complemented by Theorem~\ref{thm:a.3} which is a 
recognition devise for holomorphically induced representations.

\begin{defn} Let $q \: \bV \to M$ be a holomorphic Hilbert bundle 
on the complex manifold $M$. We write $\Gamma(\bV)$ for the space of 
holomorphic sections of $\bV$. A Hilbert subspace 
$\cH \subeq \Gamma(\bV)$ is said to have {\it 
continuous point evaluations} if all the evaluation maps 
\[ \ev_m \: \cH \to \bV_m, \quad s \mapsto s(m)\] 
are continuous and the function $m \mapsto \|\ev_m\|$ is locally bounded. 

%{\bf Do we have to ask for more? Check \cite{Ne00}. 
%Eventually we want that, locally the $B(V)$-kernel is 
%holomorphic on $M \times \oline M$, resp., that the 
%correponding kernel on $\bV^*$ is holomorphic.} 

If $G$ is a group acting on $\bV$ by holomorphic 
bundle automorphisms, $G$ acts on $\Gamma(\bV)$ by 
\begin{equation}
  \label{eq:g-act}
(g.s)(m) := g.s(g^{-1}m).   
\end{equation}
We call a Hilbert subspace $\cH \subeq \Gamma(\bV)$ with continuous 
point evaluations {\it $G$-invariant} if 
$\cH$ is invariant under the action defined by \eqref{eq:g-act} 
and the so obtained representation of $G$ on $\cH$ is 
unitary. 
\end{defn}

\begin{rem} \mlabel{rem:2.2} 
Let $q \: \bV \to M$ be a holomorphic Hilbert bundle 
over $M$. Then we can represent holomorphic sections of this 
bundle by holomorphic functions on the dual bundle $\bV^*$ 
whose fiber $(\bV^*)_m$ is the dual space  
$\bV_m^*$ of $\bV_m$.  We thus obtain an embedding 
\begin{equation}
  \label{eq:funcdual}
\Psi \: \Gamma(\bV) \to \cO(\bV^*), \quad 
\Psi(s)(\alpha_m) = \alpha_m(s(m)),  
\end{equation}
whose image consists of 
holomorphic functions on $\bV^*$ which are fiberwise linear. 

If $G$ is a group acting on $\bV$ by holomorphic 
bundle automorphisms, then $G$ also acts naturally 
by holomorphic maps on $\bV^*$ via 
$(g.\alpha_m)(v_{g.m}) := \alpha_m(g^{-1}.v_{g.m})$ 
for $\alpha_m \in \bV^*_m$. 
Therefore 
\[ \Psi(g.s)(\alpha_m) 
= \alpha_m(g.s(g^{-1}.m)) 
= (g^{-1}.\alpha_m)(s(g^{-1}.m)) 
= \Psi(s)(g^{-1}.\alpha_m) \] 
implies that $\Psi$ is equivariant with respect to the natural 
$G$-actions on $\Gamma(\bV)$ and $\cO(\bV^*)$. 
\end{rem}

\subsection{Existence of analytic vectors}

\begin{lem}
\mlabel{lem:3.1} 
If $M = G/H$ is a Banach homogeneous space with a $G$-invariant 
complex structure and $\bV = G\times_H V$ a $G$-equivariant 
holomorphic vector bundle over $M$ defined by 
a pair $(\rho,\beta)$ as in {\rm Theorem~\ref{thm:a.2}}, 
then the $G$-action on $\bV$ is analytic.  
\end{lem}

\begin{prf} The manifold 
$\bV$ is a $G$-equivariant quotient of the product 
manifold $G \times V$ on which $G$ acts analytically by left 
multiplications in the left factor. Since the quotient map 
$q \: G \times V \to \bV$ is a real analytic submersion, the 
action of $G$ on $\bV$ is also analytic. 
\end{prf}

\begin{defn} Let $M$ be a complex manifold (modeled on a locally 
convex space) and $\cO(M)$ the space of holomorphic complex-valued 
functions on $M$. We write $\oline M$ for the conjugate complex manifold. 
A holomorphic function 
$Q \: M \times \oline M \to \C$
is said to be a {\it reproducing kernel} of a Hilbert subspace 
$\cH \subeq \cO(M)$ if for each $w\in M$ the function 
$Q_w(z) := Q(z,w)$ is contained in $\cH$ and satisfies 
$$ \la f, Q_z \ra = f(z) \quad \mbox{ for } \quad z \in M, f \in \cH. $$
Then $\cH$ is called a {\it reproducing kernel Hilbert space} 
and since it is determined uniquely by the kernel $Q$, it is 
denoted $\cH_Q$ (cf.\ \cite[Sect.~I.1]{Ne00}). 

Now let $G$ be a group and 
$\sigma \: G \times M  \to M, (g,m) \mapsto g.m$ 
be a smooth right action of $G$ on $M$ by holomorphic maps. Then 
 $(g.f)(m) := f(g^{-1}.m)$ defines a unitary representation 
of $G$ on a reproducing kernel Hilbert space $\cH_Q \subeq \cO(M)$ 
if and only if the kernel $Q$ is {\it invariant}: 
$$ Q(g.z, g.w) = Q(z,w) \quad \mbox{ for } \quad z,w \in M, g \in G $$ 
(\cite[Rem.~II.4.5]{Ne00}). In this case we call 
${\cal H}_Q$ a $G$-invariant reproducing kernel Hilbert space. 
\end{defn}

\begin{lem} \mlabel{lem:2.2} 
Let $G$ be a Banach--Lie group 
acting analytically via 
\[\sigma \: G \times M \to M, \quad 
(g,m) \mapsto \sigma_g(m) \] 
by holomorphic maps on the complex manifold~$M$. 

{\rm(a)} Let $\cH \subeq \cO(M)$ be a reproducing kernel Hilbert 
space whose kernel $Q$ is a $G$-invariant holomorphic function on 
$M \times \oline M$. 
Then the elements $(Q_m)_{m \in M}$ representing the evaluation 
functionals in $\cH$ are analytic 
vectors for the representation of $G$, defined by 
$\pi(g)f := f \circ \sigma_g^{-1}$. 

{\rm(b)} Let $\bV \to M$ be a holomorphic $G$-homogeneous 
Hilbert bundle and $\cH \subeq \Gamma(\bV)$ be a $G$-invariant 
Hilbert space with continuous point evaluations. 
Then every vector of the form 
$\ev_m^*v$, $m \in M$, $v \in \bV_m$, is analytic for the 
$G$-action on $\cH$. 
\end{lem}

\begin{prf} (a) Since $Q$ is holomorphic on 
$M \times \oline M$, it is in particular real analytic. 
For $m \in M$, we have 
\[ \la \pi(g)Q_m, Q_m \ra 
= (\pi(g)Q_m)(m) = Q_m(g^{-1}m) = Q(g^{-1}m,m), \]
and since $Q$ and the $G$-action are real analytic, 
this function is analytic on $M$. Now \cite[Thm.~5.2]{Ne10b} 
implies that $Q_m$ is an analytic vector. 

(b) As in Remark~\ref{rem:2.2}, we 
realize $\Gamma(\bV)$ by holomorphic functions on the dual 
bundle $\bV^*$. We thus obtain a 
reproducing kernel Hilbert space 
$\cH_Q := \Psi(\cH) \subeq \cO(\bV^*)$, 
and since $\Psi$ is $G$-equivariant, the reproducing 
kernel $Q$ is $G$-invariant. 
For $v \in \bV_m$, evaluation in the corresponding element 
$\alpha_m := \la \cdot, v \ra \in \bV^*_m$ is given  by 
\[ s \mapsto \la s(m), v \ra = \la \ev_m(s), v\ra 
= \la s, \ev_m^*v\ra.\] 
For the corresponding $G$-invariant kernel $Q$ on $\bV^*$ 
this means that $Q_{\alpha_m} = \ev_m^*v$, so that the assertion 
follows from (a). 
\end{prf}

\subsection{The endomorphism bundle and commutants} 

The goal of this section is Theorem~\ref{thm:5.5} which connects 
the commutant of the $G$-representation on a Hilbert space $\cH_V$ 
of a holomorphically induced representation with the 
corresponding representation $(\rho, \beta)$ 
of $(H,\fq)$ on $V$. The remarkable 
point is that, under natural assumptions, these commutants are isomorphic, 
so that both representations have the same decomposition theory. 
This generalizes an important 
result of S.~Kobayashi concerning irreducibility 
criteria for the $G$-representation on $\cH_V$ (\cite{Ko68}, 
\cite[Thm.~2.5]{BH98}). 

Let $\bV = G \times_H V \to M$ be a $G$-homogeneous holomorphic 
Hilbert bundle as in Theorem~\ref{thm:a.2}. Then 
the complex manifold $M \times \oline M$ is a complex homogeneous 
space $(G \times G)/(H \times H)$, 
where the complex structure is defined by the closed subalgebra 
$\fq \oplus \oline\fq$ of $\g_\C \oplus \g_\C$. 
On the Banach space $B(V)$ we consider the norm continuous 
representation of $H \times H$ by 
\[ \tilde\rho(h_1, h_2)A = \rho(h_1)A\rho(h_2)^* \] 
and the the corresponding extension 
$\tilde\beta \: \fq \oplus \oline \fq \to \gl(B(V))$ by 
\[ \tilde\beta(x_1,x_2)A := \beta(x_1)A + A \beta(\oline{x_2})^*.\]
We write $\bL := (G \times G) \times_{H \times H} B(V)$ for the 
corresponding holomorphic Banach bundle over 
$M \times \oline M$ (Theorem~\ref{thm:a.2}). 

\begin{rem} For every $g \in G$, we have isomorphisms 
$V \to \bV_{gH} = [g,V], v \mapsto [g,v]$. 
Accordingly, we have for every pair $(g_1, g_2) \in G \times G$ an 
isomorphism 
\[ \nu \: B(V) \to B(\bV_{g_2 H}, \bV_{g_1H}), \quad 
\nu(A)[g_2,v] \mapsto [g_1,Av]. \] 
This defines a map 
\[ \gamma \: G \times G \times B(V) \to 
B(\bV) := \bigcup_{m,n \in M} B(\bV_m,\bV_n), \quad 
\gamma(g_1, g_2,A)[g_2,v] = [g_1, Av]. \] 
For $h_1, h_2 \in H$, we then have 
\begin{align*}
&\gamma(g_1h_1, g_2h_2,\tilde\rho(h_1, h_2)^{-1}A)[g_2,v] 
= \gamma(g_1h_1, g_2h_2,\tilde\rho(h_1, h_2)^{-1}A)[g_2h_2,\rho(h_2)^{-1}v] \\
&= [g_1 h_1, \rho(h_1)^{-1} A v] 
= [g_1, A v] = \gamma(g_1, g_2, A)[g_2,v], 
\end{align*}
so that $\gamma$ factors through a bijection 
$\oline\gamma \: \bL \to B(\bV).$ 
This provides an interpretation of the bundle $\bL$ as the 
{\it endomorphism bundle of $\bV$} (cf.\ \cite{BR07}). 
\end{rem}

Since the group $G$ acts (diagonally) on the bundle 
$\bL$, it makes sense to consider $G$-invariant holomorphic 
sections. 

\begin{lem} \mlabel{lem:2.7} 
The space $\Gamma(\bL)^G$ of $G$-invariant holomorphic 
sections of $\bL$ has the property that the evaluation map 
\[ \ev \: \Gamma(\bL)^G \cong C^\infty(G \times G, B(V))_{\tilde \rho, 
\tilde\beta} \to B_H(V), \quad s \mapsto \hat s(\1,\1) \] 
is injective. 
\end{lem}

\begin{prf} If $\hat s(\1,\1) = 0$, then 
the corresponding holomorphic section $s \in \Gamma(\bL)$ vanishes on the 
totally real submanifold 
\[ \Delta_M := \{ (m,m) \: m \in M \} = G(m_0, m_0)\subeq 
M \times \oline M,\] 
and hence on all of $M \times \oline M$. 
For the function $\hat s \: G \times G \to B(V)$, we 
have for $h \in H$ 
\[ \hat s(\1,\1) = (h^{-1}.\hat s)(\1,\1) 
= \hat s(h, h) = \rho(h)^{-1} \hat s(\1,\1) \rho(h), \] 
showing that $\hat s(\1,\1)$ lies in $B_H(V)$. 
\end{prf}

\begin{rem} In general, the holomorphic bundle 
$\bL$ does not have any non-zero holomorphic section, although 
its restriction to the diagonal $\Delta_M$ 
has the trivial section $R$ given by $R_{(m,m)} = \id_{\bV_m}$ 
for every $m\in M$. 

If $R \: M \times \oline M \to \bL$ is a holomorphic section, 
then we obtain for every $v \in V$ and $n \in M$ a holomorphic section 
$R_{n,v} \: M \to \bV, m \mapsto R(m,n)v$    
which is non-zero in $m$ if $R_{m,m}v \not=0$. 
Therefore $\bV$ has nonzero holomorphic sections if $\bL$ does, and this 
is not always the case. 
\end{rem}

\begin{ex} \mlabel{ex:2.11}  
(a) Let $(\pi, \cH)$ be a unitary representation of 
$G$ and $\Psi \: \cH \to \Gamma(\bV)$ be $G$-equivariant 
such that the evaluation maps $\ev_m \circ \Psi \: \cH \to \bV_m$ 
are continuous and $m \mapsto \|\ev_m\|$ is locally bounded. 
Then 
\[ R(m,n) := (\ev_m \Psi) (\ev_n \Psi)^* \in B(\bV_n, \bV_m)\]  
defines a holomorphic section of $B(\bV)$. 
Since this  assertion is local, it follows from 
the corresponding assertion for 
trivial bundles treated in \cite[Lemma~A.III.9(iii)]{Ne00}. 
The corresponding smooth function 
\[ \hat R  \: G \times G \to B(V), \quad 
\hat R(g_1, g_2) = (\ev_{g_1} \Psi) (\ev_{g_2} \Psi)^*\] 
is a $G$-invariant $B(V)$-valued 
kernel on $G \times G$ because 
\[ \ev_{gh} \Psi = \ev_h \Psi \circ \pi(g)^{-1} 
\quad \mbox{ for } \quad g,h \in G. \] 
We conclude that $\hat R \in \Gamma(\bL)^G$, so that it is completely 
determined by 
\[ \hat R(\1,\1) = (\ev_\1 \Psi) (\ev_\1 \Psi)^*\] 
(Lemma~\ref{lem:2.7}). 
In particular, $\Psi$ is completely determined by $\ev_\1 \circ \Psi$. 
This follows also from its equivariance, which leads to 
$\Psi(v)(g) = \ev_\1(g^{-1}.\Psi(v)) 
= \ev_\1\Psi(\pi(g)^{-1}v).$ 

(b) For $A, B \in B(V)$ commuting 
with $\rho(H)$ and $\beta(\fq)$ and $\hat R \in C^\infty(G \times G, 
B(V))_{\tilde\rho, \tilde\beta}$, the function 
\[ \hat R^{A,B} \: G \times G \to B(V), \quad 
\hat R^{A,B}(g_1, g_2) = A R(g_1, g_2) B\] 
is also contained in $C^\infty(G \times G, B(V))_{\tilde\rho, \tilde\beta}$, 
hence defines a holomorphic section of the bundle~$\bL$. 
\end{ex} 

\begin{defn} \mlabel{def:2.10} 
(a) For $(\rho, \beta)$ as in Theorem~\ref{thm:a.2} 
and the corresponding associated bundle $\bV = G \times_H V$, 
a unitary representation $(\pi, \cH)$ of $G$ is said to be 
{\it holomorphically induced from $(\rho,\beta)$} 
if there exists a realization 
$\Psi \: \cH \to \Gamma(\bV)$ as an invariant Hilbert 
space with continuous point evaluations whose kernel 
$R \in \Gamma(\bL)^G$ satisfies $\hat R(\1,\1) = \id_V$. 

(b) Since this condition determines $R$ by Lemma~\ref{lem:2.7}, 
it also determines the reproducing kernel space $\Psi(\cH)$ 
and its norm. 
We conclude that, for every pair $(\rho,\beta)$, 
there is at most one holomorphically induced unitary representation 
of $G$ up to unitary equivalence. 
Accordingly, we call $(\rho, \beta, V)$ {\it inducible} 
if there exists a corresponding holomorphically induced unitary 
representation of~$G$. 

If this is the case, then we use the isometric embedding 
$\ev_\1^* \: V  \to \cH$ to identify $V$ with a subspace of 
$\cH$ and note that 
the evaluation map $\ev_\1 \: \cH \to V$ corresponds to the
orthogonal projection $p_V \: \cH \to V$. 
\end{defn}

\begin{rem} \mlabel{rem:2.13} (a) If $\cH_V  \into \Gamma(\bV)$ 
is a $G$-invariant Hilbert 
space on which the representation is holomorphically induced, 
then we obtain a $G$-invariant element $Q \in B(\bV)$ by  
\[ Q(m,n) := \ev_m \ev_n^* \in B(\bV_n, \bV_m),\] 
and the relation $\ev_g = \ev_\1 \circ \pi(g)^{-1} 
= p_V \circ \pi(g)^{-1}$ yields 
\[ \hat Q(g_1, g_2) = \ev_{g_1} \ev_{g_2}^*  
= p_V \pi(g_1)^{-1} \pi(g_2)p_V = p_V \pi(g_1^{-1}g_2)p_V \] 
and in particular $\hat Q(\1,\1) = \id_V.$

(b) Suppose that $\cH \subeq \Gamma(\bV)$ is holomorphically 
induced. Since $\ev_\1$ is $H$-equivariant, the closed subspace 
$V \subeq \cH$ is $H$-invariant and the $H$-representation 
on this space is equivalent to $(\rho,V)$, hence in particular 
bounded. 

Moreover, $V \subeq \cH^\omega$ follows from 
Lemma~\ref{lem:2.2}(b). Since the evaluation maps \break $\ev_g \: \cH \to V$ 
separate the points, the analyticity of the elements 
$\ev_g^*v$ even implies that $\cH^\omega$ is dense in $\cH$. 

For $x \in \fq$  and $s \in \cH^\infty$ we have 
\[ \ev_\1\dd\pi(x) s 
= (\dd\pi(x)s)\,\hat{}(\1) 
= - L_x \hat  s(\1) 
=  \beta(x) \hat s(\1) 
=  \beta(x) \ev_\1 s.\]
Therefore $\dd\pi(\fq)$ preserves the subspace 
$\cH^\infty \cap V^\bot = \cH^\infty \cap \ker(\ev_\1)$. 
From $V \subeq \cH^\infty$, we derive 
\[ \cH^\infty = V \oplus (V^\bot \cap \cH^\infty), \] 
so that the density of $\cH^\infty$ in $\cH$ implies that 
$V = (V^\bot \cap \cH^\infty)^\bot$, and hence that 
this space is invariant under the restriction 
$\dd\pi(\oline x)$ of the adjoint $- \dd\pi(x)^*$. 
For $s_j = \ev_\1^*v_j \in V$, $j =1,2$,  we further obtain 
\begin{align*}
\la \dd\pi(\oline x)s_1, s_2 \ra 
&=  -\la \ev_\1^* v_1, \dd\pi(x) s_2 \ra 
=  -\la v_1, \ev_\1\dd\pi(x) s_2 \ra 
=  -\la v_1, \beta(x) \ev_\1 s_2 \ra \\
&=  -\la v_1, \beta(x) v_2 \ra 
=  -\la \beta(x)^*v_1,v_2 \ra,
\end{align*}
so that 
\begin{equation}\label{eq:a3} 
\dd\pi(\oline x)\res_V = - \beta(x)^*, \quad x \in \fq.
\end{equation}

Finally we observe that 
\[ (\pi(G)V)^\bot = \{ s \in \cH \: (\forall g \in G)\, \hat s(g) = 0\}
= \{0\}\] 
 implies that $\cH = \oline{\Spann(\pi(G)V)}$. 
\end{rem}

If $V$ is of the form $\oline{(\cH^\infty)^\fn}$ for a subalgebra 
$\fn \subeq \g_\C$, then it is invariant under the commutant $B_G(V)$, 
but we do not know if this is always true for holomorphically 
induced representation. 
To make the following proposition as flexible as possible, we 
assume this naturality condition of~$V$ (cf.\ Remark~\ref{rem:2.14} below).

\begin{thm} \mlabel{thm:5.5} 
Suppose that $(\pi, \cH_V)$ is holomorphically induced from 
the representation $(\rho,\beta)$ of $(H,\fq)$ on~$V$ 
and that 
\[ B_{H,\fq}(V) := 
\{ A \in B_H(V) \: (\forall x \in \fq)\, 
A \beta(x) = \beta(x) A, A^* \beta(x) = \beta(x) A^* \} \] 
is the involutive commutant of $\rho(H)$ and $\beta(\fq)$. 
If $V$ is invariant under $B_G(V)$, then the map 
\[ R \: B_G(\cH_V) \to B_{H,\fq}(V), \quad  
A \mapsto A\res_V \]
is an isomorphism of von Neumann algebras. 
\end{thm}

\begin{prf} By assumption, every $A \in B_G(\cH_V)$ preserves 
$V$, so  that 
$A\res_V$ can be identified with the operator $p_V A p_V  \in B(V)$, 
where $p_V \: \cH \to V$ is the orthogonal projection. 
Clearly $A\res_V$ commutes with each $\rho(h) = \pi(h)\res_V$. 
It also preserves the subspace $\cH^\infty$ on which it satisfies 
$\pi(x)A = A \pi(x)$ for every $x \in \g_\C$. Therefore 
$A\res_V$ commutes with $\dd\pi(\oline\fq)$ and hence with 
$\beta(\fq)$ (cf.~\eqref{eq:a3} in Remark~\ref{rem:2.13}). 

Therefore $R$ defines a homomorphism of von Neumann algebras. 
If $R(A) = 0$, then $A V = 0$ implies that 
$A \pi(G)V = \{0\}$, which leads to $A = 0$. 
Hence $R$ is injective. 

Since each von Neumann algebra is generated by orthogonal projections 
(\cite[Chap.~1, \S 1.2]{Dix96}) 
and images of von Neumann algebras under restriction maps 
are von Neumann algebras 
(\cite[Chap.~1, \S 2.1, Prop.~1]{Dix96}), 
we are done if we can show that every orthogonal 
projection in $B_{H,\fq}(V)$ is contained in the image of 
$R$. So let $P \in B_{H,\fq}(V)$. 
Then $V_1 := P(V)$ and $V_2 := (\1-P)(V)$  yields an 
$(H,\fq)$-invariant orthogonal decomposition $V = V_1 \oplus V_2$. 

Let $\hat Q \: G \times G \to B(V), (g_1, g_2) \mapsto 
p_V \pi(g_1^{-1}g_2) p_V$ be the natural kernel function 
defining the inclusion $\cH_V \into \Gamma(\bV)$ and consider the 
$G$-invariant kernel 
\[  \hat R(g_1, g_2) := P \hat Q(g_1, g_2) (\1 - P). \]
According to Example~\ref{ex:2.11}(b), 
it is contained in $C^\infty(G \times G, B(V))_{\tilde\rho,\tilde\beta}$, 
hence defines an element $R \in \Gamma(\bL)^G$. 
In view of 
\[  \hat R(\1,\1) = P Q(\1,\1) (\1 - P) = P(\1 - P) = 0, \]
this section vanishes in the base point, hence on all of 
$M \times \oline M$ (Lemma~\ref{lem:2.7}). We conclude that 
\[ 0 = \hat R(g_1, g_2) 
= P  p_V \pi(g_1^{-1}g_2) p_V (\1-P) 
= P \pi(g_1^{-1}g_2) (\1 - P), \] 
so that, for every $g \in G$, we have 
$P \pi(g)(\1-P) = 0.$ 
This leads to $P  \cH_{V_2} = \{0\}$, and hence to 
$\cH_{V_1} \bot \cH_{V_2}$. We derive that 
$\cH_V = \cH_{V_1} \oplus \cH_{V_2}$ is an orthogonal 
direct sum. Therefore the 
orthogonal projection $\tilde P \in B_G(\cH_V)$ onto $\cH_{V_1}$  
leaves $V$ invariant and satisfies 
$\tilde P\res_V = P$. This proves that $R$ is surjective. 
\end{prf}

\begin{rem} The preceding proof shows that, under the assumptions of 
Theorem~\ref{thm:5.5}, the range of the injective map 
\[ \ev \: \Gamma(\bL)^G \to B_H(V), \quad R \mapsto \hat R(\1,\1)\] 
contains $B_{H,\fq}(V)$. If $\beta(\fq) = \beta(\fh_\C)$, then 
$B_{H,\fq}(V) = B_H(V)$, so that we obtain a linear isomorphism 
$\Gamma(\bL)^G\cong B_H(V)$. 
\end{rem}

The preceding theorem has quite remarkable 
consequences because it implies that the representations 
$(\pi, \cH_V)$ and $(\rho,V)$ decompose in the same way. 

\begin{cor} \mlabel{cor:commutant} 
Suppose that $(\pi, \cH_V)$ is holomorphically induced from 
$(\rho,\beta)$, that $V$ is $B_G(V)$-invariant,  
and that $\beta(\fq) = \beta(\fh_\C)$, so that 
$B_H(V) = B_{H,\fq}(V)$. 
Then the $G$-representation $(\pi, \cH_V)$ 
has any of the following properties if and only if 
the $H$-representation $(\rho,V)$ does. 
\begin{description}
\item[\rm(i)] Irreducibility. 
\item[\rm(ii)] Multiplicity freeness. 
\item[\rm(iii)]  Type $I$, $II$ or $III$. 
\item[\rm(iv)] Discreteness, i.e., being a direct 
sum of irreducible representations. 
\end{description}
\end{cor}

\begin{prf} (i) follows from the fact that, according to Schur's Lemma, 
irreducibility means that the commutant equals $\C \1$. 
  
(ii) is clear because multiplicity freeness means that the commutant 
is commutative. 

(iii) is clear because the type of a representation is defined as 
the type of its commutant as a von Neumann algebra. 

(iv) That a unitary representation decomposes discretely 
means that its commutant is an $\ell^\infty$-direct sum of 
factors of type $I$. Hence the $G$ representation on $\cH_V$ has
this property if only if the $H$-representation on $V$~does.
\end{prf}

Corollary~\ref{cor:commutant}(i) is a version of 
S.~Kobayashi's Theorem in the Banach context 
(cf.\ \cite{Ko68}). 

\begin{prob} Theorem~\ref{thm:5.5} should 
also be useful to derive direct integral decompositions 
of the $G$-representation on $\cH_V$ from direct integral 
decompositions of the $H$-representation on $V$. 

Suppose that the bounded unitary representation 
$(\rho, V)$ of $H$ is of type I and holomorphically inducible. 
Then it is a direct integral of irreducible representations 
$(\rho_j, V_j)$. Are these irreducible representations of 
$H$ also inducible?   
\end{prob}

\begin{cor} \mlabel{cor:2.15} Suppose that the $G$-representations 
$(\pi_1, \cH_{V_1})$, resp., $(\pi_2, \cH_{V_2})$ 
are holomorphically induced from the $(H,\fq)$-representations 
$(\rho_1, \beta_1, V_1)$, resp., $(\rho_2, \beta_2, V_2)$. 
Then any unitary isomorphism 
$\gamma \: V_1 \to V_2$ of $(H,\fq)$-modules extends uniquely to a 
unitary equivalence $\tilde\gamma \: \cH_{V_1} \to \cH_{V_2}$. 
\end{cor}

\begin{prf} Since $\gamma$ is $(H,\fq)$-equivariant, we have a 
well-defined $G$-equivariant bijection 
\[ \tilde \gamma \: \Gamma(\bV_1) \cong 
C^\infty(G,V_1)_{\rho_1, \beta_1} \to 
C^\infty(G,V_2)_{\rho_2, \beta_2}, \quad 
f \mapsto \gamma \circ f \] 
obtained from a corresponding isomorphism 
$[(g,v)] \mapsto [(g,\gamma(v))]$ of holomorphic $G$-bundles. 
Therefore $\tilde\gamma(\cH_{V_1})$ is an invariant Hilbert space 
with continuous point evaluations. The corresponding 
$G$-invariant kernel $\hat Q \in C^\infty(G \times G, B(V_2))$ 
satisfies 
\[ \hat Q(\1,\1) = \gamma \circ \id_{V_1} \circ \gamma^* 
= \gamma \gamma^* = \id_{V_2}.\] 
This implies that $\tilde\gamma(\cH_{V_1}) 
= \cH_{V_2}$ and that the map 
$\tilde \gamma \: \cH_{V_1} \to \cH_{V_2}$ is unitary 
because both spaces have the same reproducing kernels 
(cf.\ Definition~\ref{def:2.10}). 

For $v \in V_1$, we further note that 
$(\tilde\gamma \ev_\1^*v)(\1) 
= \gamma \ev_\1 \ev_\1^*v = \gamma v$, so that 
$\tilde\gamma$ extends $\gamma$, when considered as a map 
on the subspace $V_1 \cong \ev_\1^*V_1 \subeq \cH_1$. 
\end{prf}

\subsection{Realizing unitary representations by holomorphic sections} 

We conclude this section with a result that helps to realize 
certain subrepresentations of unitary representations in 
spaces of holomorphic sections. We continue in the setting 
of Section~\ref{sec:1}, where $M = G/H$ is a Banach homogeneous 
space with a complex structure defined by the subalgebra 
$\fq \subeq \g_\C$. 

From the discussion in Remark~\ref{rem:2.13}, we know that 
every holomorphically induced representation 
satisfies the assumptions (A1/2) in the theorem below, 
which is our main tool to prove that a given unitary representation 
is holomorphically induced. 

\begin{thm}\mlabel{thm:a.3} 
Let $(\pi, \cH)$ be a continuous unitary representation of $G$ 
and $V \subeq \cH$ be a closed subspace satisfying the 
following conditions: 
\begin{description}
\item[\rm(A1)] $V$ is $H$-invariant and the representation 
$\rho$ of $H$ on $V$ is bounded. In particular,  
$\dd\pi \res_{\fh}\: \fh \to \gl(V)$ 
defines a continuous homomorphism 
of Banach--Lie algebras. 
\item[\rm(A2)] The exists a subspace $\cD_V \subeq V \cap \cH^\infty$ dense 
in $V$ which is invariant under $\dd\pi(\oline\fq)$, 
the operators $\dd\pi(\oline\fq)\res_{\cD_V}$ are bounded, 
 and the so obtained representation of 
$\oline\fq$ on $V$ defines a continuous morphism 
of Banach--Lie algebras 
$$\beta \: \fq \to \gl(V), \quad x \mapsto 
-(\dd\pi(\oline x)\res_V)^*.$$ 
\end{description}
Then the following assertions hold: 
\begin{description}
\item[\rm(i)] $\beta$ is an extension of $\rho$ defining on 
$\bV := G \times_H V$ the structure of a complex 
Hilbert bundle. 
\item[\rm(ii)] If $p_V \: \cH \to \cH$ denotes 
the orthogonal projection to~$V$, then 
$$ \Phi \: \cH \to C(G,V)_\rho, \quad \Phi(v)(g) := p_V(\pi(g)^{-1}v) $$ 
maps $\cH$ into $C^\infty(G,V)_{\rho,\beta} \cong \Gamma(\bV)$, 
and we thus obtain a $G$-equivariant unitary isomorphism 
of the closed subspace 
$\cH_V := \oline{\Spann \pi(G)V}$ 
with the representation holomorphically induced from $(\rho, \beta)$. 
\item[\rm(iii)] $V \subeq \cH^\omega$. 
\end{description}
\end{thm}

\begin{prf} (i) For $x \in \fh$, we have 
$\beta(x) = -\dd\rho(x)^*\ = \dd\rho(x),$ 
and it is also easy to see that 
$\beta(\Ad(h)x) = \pi(h) \beta(x)\pi(h)^{-1}$ for 
$h \in H$ and $x \in \fq$. Therefore $\beta$ is an extension 
of $\rho$, and we can use 
Theorem~\ref{thm:a.2} to see that 
$\beta$ defines the structure of a holomorphic Hilbert 
bundle on $\bV$. 

(ii) Clearly, $\Phi(\cH^\infty) \subeq C^\infty(G,V)_\rho$. 
For $v \in \cD_V$, $w \in \cH^\infty$ and $x \in \fq$, we further 
derive from (A2) that 
\begin{align*}
\la p_V(\dd\pi(x)w), v \ra 
&= \la \dd\pi(x)w, v \ra = \la w, \dd\pi(-\oline x)v \ra 
= \la p_V(w), \dd\pi(-\oline x)v \ra 
= \la \beta(x)p_V(w),v \ra, 
\end{align*}
so that the density of $\cD_V$ in $V$ implies 
$$ p_V \circ \dd\pi(x) =  \beta(x) \circ p_V\res_{\cH^\infty} 
\quad \mbox{ for } 
\quad x \in \fq. $$
For $v \in \cH^\infty$ and $x \in \fq$, we now obtain 
\begin{equation}
  \label{eq:lx-rel}
\big(L_x \Phi(v)\big)(g) 
= -p_V(\dd\pi(x)\pi(g)^{-1}v) 
= -\beta(x)p_V(\pi(g)^{-1}v)
= -\beta(x)\Phi(v)(g). 
\end{equation}
This means that 
$\Phi(\cH^\infty) \subeq C^\infty(G,V)_{\rho,\beta}$. 
Writing $\Gamma_c(\bV)$ for the space of continuous sections of 
$\bV$, we also obtain a map 
$\tilde\Phi \: \cH \to \Gamma_c(\bV)$ which is continuous if 
$\Gamma_c(\bV)$ is endowed with the compact open topology. 
As $\Gamma(\bV)$ is closed in $\Gamma_c(\bV)$ with respect to this topology 
(\cite[Cor.~III.12]{Ne01}), 
$\tilde\Phi(\oline{\cH^\infty}) \subeq \Gamma(\bV)$, 
resp.,  $\Phi(\oline{\cH^\infty}) \subeq C^\infty(G,V)_{\rho,\beta}$. 

Clearly $\Phi(v) = 0$ is equivalent to 
$v \bot \pi(G)V$, so that $(\ker \Phi)^\bot = \cH_V$. 
Since (A2) implies that $V \subeq \oline{\cH^\infty}$, the same 
holds for $\Spann(\pi(G)V)$. This shows that 
\[ \Phi(\cH) = \Phi(\cH_V) 
= \Phi(\oline{\cH^\infty}) \subeq C^\infty(G,V)_{\rho,\beta}.\] 
The corresponding kernel $\hat Q \: G \times G \to B(V)$ is given by 
\[ \hat Q(g_1, g_2) = \ev_{g_1} \ev_{g_2}^* 
= p_V \pi(g_1)^{-1} (p_V \circ \pi(g_2)^{-1})^* 
= p_V \pi(g_1^{-1}g_2) p_V,\] 
so that we have in particular $\hat Q(\1,\1) = \id_V$. 

(iii) To see that $V$ consists of analytic vectors, 
we may w.l.o.g.\ assume that $\cH = \cH_V$ and hence that 
$\cH \subeq \Gamma(\bV)$ is holomorphically induced and that 
$\Phi = \id_\cH$. 
Let $\ev_{\1 H} \: \cH \to \bV_{\1 H} \cong V$ be the evaluation map. 
The corresponding map 
\[ \ev_\1 \: C^\infty(G,V)_{\rho,\beta} \to V \] 
is simply given by evaluation in $\1 \in G$. 
Now $\ev_\1\res_V \: V \to V$ is the identity, so that 
$\ev_\1^* \: V\to \cH$ is simply the isometric inclusion. Hence 
the analyticity of $v = \ev_\1^*v = \ev_{\1 H}^*v$ 
follows from Lemma~\ref{lem:2.2}. 
\end{prf} 

\begin{rem} \mlabel{rem:2.14} 
Suppose that there exist subalgebras $\fp^\pm \subeq \g_\C$ with 
$\fq = \fh_\C \rtimes \fp^+$ and $\Ad(H)\fp^\pm = \fp^\pm$. 
Then, for every unitary representation 
$(\pi, \cH)$ of $G$, the closed subspace 
$V := \oline{(\cH^\infty)^{\fp^-}}$ 
is invariant under $H$ and $B_G(\cH)$. 
If the representation $(\rho, H)$ on $V$ is bounded, 
then the dense subspace $\cD_V := (\cH^\infty)^{\fp^-}$ 
satisfies (A2) if we put $\beta(\fp^+) = \{0\}$. 
Theorem~\ref{thm:a.3} now implies that 
$V \subeq \cH^\omega$, so that we see in particular that 
$V = (\cH^\infty)^{\fp^-}$ 
is a closed subspace of $\cH$. 
Moreover, Corollary~\ref{cor:commutant} applies. 
\end{rem}

The following remark sheds some extra light on condition (A2). 

\begin{rem}
Let $\cH \subeq \Gamma(\bV)$ be a $G$-invariant Hilbert subspace 
with continuous point evaluations, $m \in M$ and 
$V := \oline{\im(\ev_{m}^*)} \subeq \cH$. 
Then $V$ is a $G_m$-invariant closed subspace of $\cH$ and 
\[ V^\bot = \im(\ev_m^*)^\bot = \ker(\ev_m) 
= \{ s \in \cH \: s(m) =0\}. \] 

The action of $G$ on $\bV$ by holomorphic bundle automorphisms 
leads to a homomorphism $\dot\sigma_{\bV} \: \g_\C \to \cV(\bV)$ 
and each $x \in \fq_m$ thus leads to a linear vector field 
$-\beta_m(x)$ on the fiber $\bV_m$. Passing to derivatives in the formula 
$(g.s)(m) := g.s(g^{-1}m)$, we obtain for $x \in \fq_m$ 
\[ (x.s)(m) = -\beta_m(x)\cdot s(m).\] 
In particular, $V^\bot \cap \cH^\infty$ is invariant under 
the derived action of $\fq_m$, so that one can expect that the adjoint 
operators coming from $\oline{\fq_m}$ act on $\bV_m$

As we have seen in Lemma~\ref{lem:2.2}(b), the 
subspace $\im(\ev_m^*)$ of $\cH$ consists of smooth vectors, so that 
$V \cap \cH^\infty$ is dense in $V$. 
\end{rem}

\begin{ex} \mlabel{ex:grass} 
We consider the identical representation of 
$G = \U(\cH)$ on the complex Hilbert space $\cH$. 
Let $\cK$ be a closed subspace of $\cH$. Then 
the subgroup $Q := \{ g \in \GL(\cH) \: g\cK = \cK\}$ 
is a complex Lie subgroup of $\GL(\cH)$ and the Gra\ss{}mannian 
$\Gr_\cK(\cH) := \GL(\cH)\cK \cong \GL(\cH)/Q$ carries the 
structure of a complex homogeneous space on which the unitary 
group $G = \U(\cH)$ acts transitively and which is isomorphic 
to $G/H$ for $H := \U(\cH)_\cK \cong \U(\cK) \oplus \U(\cK^\bot)$. 

Writing elements of $B(\cH)$ as $(2 \times 2)$-matrices according to 
the decomposition $\cH = \cK \oplus \cK^\bot$, we have 
\[ \fq = \Big\{ \pmat{ a & b \\ 0 & d} \: a \in B(\cK), b \in B(\cK^\bot, \cK), 
d \in B(\cK^\bot)\Big\},\] 
and $\gl(\cH) = \fq \oplus \fp^-$ holds for 
$\fp^- = \Big\{ \pmat{ 0 & 0 \\ c & 0} \: c \in B(\cK, \cK^\bot)\Big\}.$ 

The representation of $\U(\cH)$ on $\cH$ is bounded with 
$V := \cH^{\fp^-} = \cK^\bot$, and the representation of 
$H \cong \U(\cK) \oplus \U(\cK^\bot)$ on this space is bounded. 
In view of Theorem~\ref{thm:a.3}, the 
canonical extension $\beta \: \fq \to \gl(V), \beta(x) := (x^*\res_V)^*$ 
now leads to a holomorphic vector bundle 
$\bV := \GL(\cH) \times_Q V \cong G \times_H V$ and a $G$-equivariant 
realization $\cH \into \Gamma(\bV)$. 

In this sense every Hilbert space can be realized as a space of 
holomorphic sections of a holomorphic vector bundle over any Gra\ss{}mannian 
associated to~$\cH$. Note that $\U(\cH)_\cK = \U(\cH)_V$ shows that 
$\Gr_\cK(\cH) \cong G/H$ can be identified in a natural way with
$\Gr_V(\cH)$. 
\end{ex} 

\begin{rem} Let $(\pi, \cH)$ be a smooth unitary representation 
of the Lie group $G$ and $V \subeq \cH$ be a closed $H$-invariant subspace. 
We then obtain a natural $G$-equivariant map 
$\eta \: G/H \to \Gr_V(\cH), gH \mapsto \pi(g)V$. 
If this map is holomorphic, then we can pull back the natural 
bundle $\bV \to \Gr_V(\cH)$ from Example~\ref{ex:grass} and 
obtain a realization of $\cH$ in $\Gamma(\eta^*\bV)$. This works 
very well if the representation 
$(\pi, \cH)$ is bounded because in this case $\pi \: G \to \U(\cH)$
is a morphism of Banach--Lie groups, but if 
$\pi$ is unbounded, then it seems difficult to verify 
that $\eta$ is smooth, resp., holomorphic. 

If (A1/2) in Theorem~\ref{thm:a.3} are satisfied, then 
$V \subeq \cH^\omega \subeq \cH^\infty$ implies that, 
the operators $\dd\pi(x)$, $x \in \g$, are defined on $V$. 
Since they are closable, the graph of these restrictions is closed, 
which implies that the restrictions $\dd\pi(x)\res_V \: V \to \cH$ 
are continuous linear operators. We thus obtain a 
natural candidate for a tangent map 
\[ T_H(\eta) \: T_H(G/H) \cong \g/\fh \to 
T_V(\Gr_V(\cH)) \cong B(V,V^\bot), \quad 
x \mapsto (\1-p_V)\dd\pi(x) p_V.\] 

For the special case where $\dim V = 1$, we have 
$V= \C v_0$ for a smooth vector $v_0$, and since the projective orbit map 
$G \to \bP(\cH) \cong \Gr_V(\cH), g \mapsto [\pi(g)v_0]$ is smooth, 
the induced map $\eta \:  G/H \to \bP(\cH)$ is smooth as well.
This construction is the key idea 
behind the theory of coherent state representations 
(\cite{Od88}, \cite{Od92}, \cite{Li91}, \cite{Ne00}, \cite{Ne01}), 
where one uses a holomorphic map $\eta \: G/H \to \bP(\cH^*)$ of a 
complex homogeneous space $G/H$ to realize a unitary representation 
$(\pi, \cH)$ of $G$ in the space of holomorphic sections 
of the line bundle $\eta^*\bL$, where 
$\bL \to \bP(\cH^*)$ is the canonical bundle on the dual projective
space with $\Gamma(\bL) \cong \cH$. 
\end{rem}

\section{Realizing positive energy representations}  \mlabel{sec:3}

In this section we fix an element $d \in \g = \L(G)$ for which 
the one-parameter group $e^{\R \ad d} \subeq \Aut(\g)$ is bounded, 
i.e., preserves an equivalent norm. We call such elements 
{\it elliptic}. Then $H := Z_G(d)$ is a Lie subgroup and if 
$0$ is isolated in $\Spec(D)$, then $G/H$ carries a natural complex 
structure. The class of representations which one may expect 
to be realized by holomorphic sections of Hilbert bundles 
$\bV$ over $G/H$ is the class of {\it positive energy representations}, 
which is defined by the condition that the selfadjoint operator 
$-i\dd\pi(d)$ is bounded below. 

\subsection{The splitting condition} 

Let $d \in \g$ be an elliptic element. Then 
\[ H = Z_G(\exp \R d) = Z_G(d) 
= \{  g \in G \: \Ad(g)d = d \} \] 
is a closed subgroup of $G$, not necessarily connected, 
 with Lie algebra $\fh = \z_\g(d) = \ker(\ad d)$. 
Since $\g$ contains arbitrarily small $e^{\R \ad d}$-invariant 
$0$-neighborhoods $U$, there exists such an open $0$-neighborhood 
with $\exp_G(U) \cap H = \exp_G(U \cap \L(H)).$ 
Therefore $H$ is a Lie subgroup of $G$, i.e., a Banach--Lie group 
for which the inclusion $H \into G$ is a topological embedding. 

Our assumption implies that $\alpha_t := e^{t\ad d}$ defines an 
equicontinuous one-paramter group of automorphisms of the 
complex Banach--Lie algebra $\g_\C$. For 
$\delta > 0$, we consider the Arveson spectral subspace 
$\fp^+ := \g_\C([\delta,\infty[)$ 
(cf.\ Definition~\ref{def:arv}). 
Applying Proposition~\ref{prop:spec-add} to the Lie bracket 
$\g_\C \times \g_\C \to \g_\C$, we see that 
$\fp^+$ is a closed complex subalgebra. 
For $f \in L^1(\R)$, $\alpha(f) := \int_\R f(t) \alpha_t\, dt$ 
and $x \in \g_\C$, 
the relations $\oline{\alpha(f)x} = \alpha(\oline f)\oline x$ and 
$\hat{\oline f}(\xi) = \oline{\hat f(-\xi)}$ imply that 
$\fp^- := \oline{\fp^+} = \g_\C(]-\infty, -\delta])$. 
To make the following constructions work, we assume the 
{\it splitting condition:} 
\begin{equation}
  \label{eq:splitcond}
\g_\C = \fp^+ \oplus \h_\C \oplus \fp^-.\tag{SC}
\end{equation}
In view of Lemma~\ref{lem:a.17}, it is satisfied for some 
$\delta > 0$ if and only if $0$ is isolated in $\Spec(\ad d)$. 

Since $\Ad(H)$ commutes with $e^{\R \ad d}$, 
the closed subalgebras $\fp^\pm \subeq \g_\C$ are invariant 
under $\Ad(H)$ and $e^{\R \ad d}$. 
Now $\g \cap (\fp^+ \oplus \fp^-)$ is a closed complement for 
$\fh$ in $\g$, so that $M := G/H$ carries the structure of a 
Banach homogeneous space and 
$\fq := \fh_\C + \fp^+ \cong \fp^+ \rtimes \fh_\C$ 
defines a $G$-invariant complex manifold 
structure on $M$ (cf.\ Section~\ref{sec:1}). 

\begin{rem} (a) For bounded derivations of compact 
$L^*$-algebras similar splitting conditions have been used 
by Belti\c{t}\u{a} in \cite{Bel03} to obtain K\"ahler polarizations 
of coadjoint orbits. In \cite{Bel04} this is extended to 
bounded normal derivations of a complex Banach Lie algebra. 

(b) If $\g$ is a real Hilbert--Lie algebra, then one can use
 spectral measures to obtain natural 
complex structures on $G/H$ even if the splitting condition is 
not satisfied, i.e., $0$ need not be isolated in the spectrum 
of $\ad d$ (\cite[Prop.~5.4]{BRT07}). 
\end{rem}

\begin{ex} If $\alpha$ factors through an action of 
the circle group $\T = \R/2\pi\Z$, then the Peter--Weyl Theorem 
implies that the sum $\sum_{n \in \Z} \g_\C^n$ of the corresponding 
eigenspaces $\g_\C^n := \ker(\ad d - in \1)$ 
is dense in $\g_\C$ with $\fh_\C = \g_\C^0$. Since the operator 
$\ad d$ is bounded, only finitely many $\g_\C^n$ are non-zero, so that 
we actually have $\g_\C = \sum_{n \in \Z} \g_\C^n$. 
From $[\g_\C^n, \g_\C^m] \subeq \g_\C^{n+m}$ it follows 
that $\fp^\pm := {\sum_{\pm n > 0} \g_\C^n}$ are closed subalgebras 
for which $\g_\C = \fp^+  \oplus \fh_\C \oplus \fp^-$ is direct. 
In this case the splitting condition is always satisfied 
and $\Spec(D) \subeq i\Z$. 

If, conversely, $d \in \g$ is an element for which the complex 
linear extension of $\ad d$ to $\g_\C$ is diagonalizable 
with finitely many eigenvalues in $i\Z$, then $e^{\R \ad d} 
\subeq \Aut(\g_\C)$ is compact, hence preserves a compatible norm. 
An important special situation, where we have all this 
structure are hermitian Lie groups (cf.\ \cite{Ne11}). 
In this case we simply have $\fp^\pm = \g_\C^{\pm 1}$. 
\end{ex}

\subsection{Positive energy representations} 

\begin{lem} \mlabel{lem:c.1} 
Let $\gamma \: \R \to \U(\cH)$ be a strongly continuous 
unitary representation and $A = A^* = -i\gamma'(0)$ be its selfadjoint 
generator, so that $\gamma(t) = e^{itA}$ in terms of measurable functional 
calculus. Then the following assertions hold: 
\begin{description}
\item[\rm(i)]  For each $f \in L^1(\R)$, we have 
$\gamma(f) = \hat f(A),$
where $\hat f(x) := \int_\R e^{ixy} f(y)\, dy$ is the Fourier transform 
of $f$. 
\item[\rm(ii)]  Let $P \: \fB(\R) \to B(\cH)$ be the unique 
spectral measure with $A = P(\id_\R)$. 
Then, for every closed subset $E \subeq \R$, the 
range $P(E)\cH$ coincides with the Arveson spectral subspace 
$\cH(E)$. 
\end{description}
\end{lem}

\begin{prf} Since the unitary representation $(\gamma,\cH)$ is a direct sum 
of cyclic representation, it suffices to prove the assertions for 
cyclic representations. Every cyclic representation of 
$\R$ is equivalent to the representation on 
some space $\cH = L^2(\R,\mu)$, where $\mu$ is a Borel probability 
measure on $\R$ and $(\gamma(t)\xi)(x) = e^{itx}\xi(x)$ 
(see \cite[Thm.~VI.1.11]{Ne00}). 

(i) We have $(A\xi)(x) = x\xi(x)$, so that 
$(\hat f(A)\xi)(x) = \hat f(x)\xi(x)$. On the other hand, we have for 
$f \in L^1(\R)$ in the space $\cH = L^2(\R,\mu)$ the relation 
\[ (\gamma(f)\xi)(x) 
= \int_\R f(t) e^{itx}\xi(x)\, dt = \hat f(x)\xi(x). \] 

(ii) see \cite[p.~226]{Ar74}. 
\end{prf}

\begin{prop} \mlabel{prop:c.3} 
Let $(\pi, \cH)$ be a smooth unitary representation 
of the Banach--Lie group $G$, $d \in \g$ be elliptic, 
and $P \: \fB(\R) \to \cL(\cH)$ be the spectral measure 
of the unitary one-parameter group 
$\pi_d(t) := \pi(\exp_G td)$. 
Then the following assertions hold: 
\begin{description}
\item[\rm(i)] $\cH^\infty$ carries a Fr\'echet structure for which 
$\pi_d(t)_{t \in \R}$ defines a continuous equicontinuous action of 
$\R$ on $\cH^\infty$. In particular, $\cH^\infty$ is invariant under 
$\pi_d(f)$ for every $f \in L^1(\R)$. 
\item[\rm(ii)] For every closed subset $E \subeq \R$, we have 
$\cH^\infty(E) = (P(E)\cH) \cap \cH^\infty$ for the corresponding 
spectral subspace. 
\item[\rm(iii)]  For every open subset $E \subeq \R$, 
$(P(E) \cH) \cap \cH^\infty$ is dense in $P(E)\cH^\infty$. 
More precisely, there exists a sequence $(f_n)_{n \in \N}$ 
in $L^1(\R)$ for which $\pi_d(f_n) \to P(E)$ in the 
strong operator topology and $\supp(\hat f_n) \subeq E$, so that 
$\pi_d(f_n)v \in \cH^\infty \cap P(E)\cH^\infty$ for every $v \in \cH^\infty$. 
\item[\rm(iv)]  %\mlabel{prop:4.3} 
For closed subsets $E, F \subeq \R$, 
the Arveson spectral subspaces $\g_\C(F)$ satisfies 
\begin{equation}
  \label{eq:shift}
\dd\pi(\g_\C(F))\big(\cH^\infty \cap P(E)\cH\big) \subeq P(\oline{E+ F})\cH. 
\end{equation}
\end{description}
\end{prop}  

\begin{prf} (i) We may w.l.o.g.\ assume that the norm on 
$\g$ is invariant under $e^{\R \ad d}$. 
On $\cH^\infty$ we consider the Fr\'echet topology 
defined by the seminorms 
$$ p_n(v) := \sup \{ \|\dd\pi(x_1)\cdots \dd\pi(x_n)v\| \: 
x_i \in \g, \|x_i\| \leq 1\} $$
with respect to which the action of $G$ on $\cH^\infty$ is smooth  
(cf.~\cite[Thm.~4.4]{Ne10a}). In particular, the bilinear map 
\begin{equation}
  \label{eq:applic}
\g_\C \times \cH^\infty \to \cH^\infty, \quad 
(x,v) \mapsto \dd\pi(x) v
\end{equation}
is continuous because it can be obtained as a 
restriction of the tangent map of the $G$-action. 
 
In view of the relation 
$\pi_d(t) \dd\pi(x) \pi_d(t)^{-1} = \dd\pi(e^{t \ad d}x)$ 
for $t \in \R, x \in \g,$ 
the isometry of $e^{t\ad d}$ on $\g$ implies that 
the seminorms $p_n$ on $\cH^\infty$ 
are invariant under $\pi_d(\R)$. Since the $\R$-action 
on $\cH^\infty$ defined by the operators $\pi_d(t)$ is smooth, 
hence in particular continuous, we obtain with Definition~\ref{def:arv} 
an algebra homomorphism 
$$ \pi_d \: (L^1(\R), *) \to \End(\cH^\infty), \quad 
f \mapsto \int_\R f(t)\pi_d(t)\, dt $$  
(cf.\ Definition~\ref{def:arv} below), and this implies (i). 

(ii) Since $\cH^\infty(E) = \cH(E) \cap \cH^\infty$ follows immediately 
from the definition of spectral subspaces (Remark~\ref{rem:a.6}), 
this assertion is a consequence of Lemma~\ref{lem:c.1}(ii). 

(iii) We write the open set $E$ as an increasing 
union of compact subsets $E_n$ 
%\[E_n := \Big\{ t \in E \Big| |t| \leq n, \mathrm{dist}(t,E^c) \geq 
%{\textstyle\frac{1}{n}}\Big\} \] 
and observe that $\bigcup_n P(E_n) \cH$ is dense in $P(E)\cH$.  
For every $n$, there exists a compactly supported function 
$h_n \in C^\infty_c(\mathbb R,\mathbb R)$ such that 
\[ \supp(h_n) \subseteq E, \quad 0 \leq h_n \leq 1, \quad 
\mbox{ and } \quad h_n\big|_{E_n} = 1.\] 
 Let $f_n \in \cS(\mathbb R)$ with 
$\hat f_n = h_n$. Then 
$\pi_d(f_n) = \hat f_n(-i\gamma'(0)) = h_n(-i\gamma'(0))$ 
(Lemma~\ref{lem:c.1}(i)) and consequently
\[ P(E_n)\cH \subseteq \pi_d(f_n)\cH \subseteq 
P(E)\cH. \]
Therefore the subspace 
$\pi_d(f_n)\cH^\infty$ of $\cH^\infty$ is contained 
in $P(E)\cH$. 
If $w = P(E)v$ for some $v \in \cH^\infty$
then 
\[ \pi_d(f_n)w  = \pi_d(f_n)P(E)v = \pi_d(f_n)v
\in \cH^\infty\] 
and 
\[ \|\pi_d(f_n)w-w\|^2 
=  \|h_n(-i\pi_d'(0))w -w\|^2 \leq \|P(E\backslash E_n)w\|^2 \to 0\] 
from which it follows that $\pi_d(f_n)w \to w$. 

(iv) This follows  from the continuity of 
\eqref{eq:applic}, Proposition~\ref{prop:spec-add} and (ii). 
\end{prf}

Results of a similar type as Proposition~\ref{prop:c.3}(iv)  
and the more universal Proposition~\ref{prop:spec-add} in the 
appendix are well known in the context of bounded operators 
(cf.\ \cite{FV70}, \cite{Ra85}, \cite[Prop.~1.1, Cor.~1.2]{Bel04}). 
Arveson also obtains variants for automorphism groups of operator algebras 
(\cite[Thm.~2.3]{Ar74}). 

\begin{rem} Combining Lemma~\ref{lem:c.1}(i) with 
Proposition~\ref{prop:c.3}(i), we derive that 
the subspace $\cH^\infty$ of $\cH$ is invariant under all operators 
$P(\hat f) = \hat f(-i\dd\pi(d))$ for $f \in L^1(\R)$. This implies 
in particular to the operators $P(h)$, $h \in \cS(\R)$, but not 
to the spectral projections $P(E)$.   If $E \subeq \R$ is open, 
Proposition~\ref{prop:c.3}(iii) provides a suitable approximate 
invariance. 
\end{rem}

The following proposition is of key importance for the following. 
It contains the main consequences of Arveson's spectral theory  
for the actions on $\g_\C$ and $\cH^\infty$. 

\begin{prop} \mlabel{prop:6.2} If $d \in \g$ is elliptic with 
$0$ isolated in $\Spec(\ad d)$, then for any 
smooth positive energy representation 
 $(\pi, \cH)$ of $G$,  the $H$-invariant subspace 
$V := \oline{(\cH^\infty)^{\fp^-}}$ satisfies 
$\cH = \oline{\Spann(\pi(G)V)}$. 
\end{prop} 

\begin{prf} First we show that $V \not=\{0\}$ whenever 
$\cH\not=\{0\}$. Let 
\[ s := \inf(\Spec(-i\dd\pi(d))) > -\infty . \]
For some $\eps \in ]0,\delta[$, we consider the closed subspace 
\begin{equation}
  \label{eq:vdef}
W := P([s,s+ \eps[) \cH = P(]s-\eps, s+ \eps[) \cH, 
\end{equation}
where $P \: \fB(\R) \to B(\cH)$ is the spectral measure of~$\pi_d$. 
Then Proposition~\ref{prop:c.3} implies that  
$W^\infty :=  W\cap \cH^\infty$ is dense in $W$ 
and that 
\[  \dd\pi(\fp^-) W^\infty 
\subeq P(]-\infty, s+ \eps - \delta])\cH = \{0\},\]
which leads to $\{0\}\not= W \subeq V$. 

Applying the preceding argument to  the positive energy 
representation on the orthogonal complement of 
$\cH_V := \oline{\Spann \pi(G)V}$, the relation 
$V \cap \cH_V^\bot=\{0\}$ implies that 
$\cH_V^\bot=\{0\}$, and hence that $\cH = \cH_V$. 
\end{prf}

\begin{thm} \mlabel{thm:6.2} If $d \in \g$ is elliptic with 
$0$ isolated in $\Spec(\ad d)$ and $(\pi,\cH)$ 
is a smooth positive energy representation 
for which the $H$-representation $\rho(h) := \pi(h)\res_V$ 
on $V := \oline{(\cH^\infty)^{\fp^-}}$  
is bounded, then $(\pi, \cH)$ is holomorphically induced from 
the representation $(\rho,\beta)$ of $(H,\fq)$ on 
$V$ defined by $\beta(\fp^+) = \{0\}$. In particular, 
$V$ consists of analytic vectors. 
\end{thm} 

\begin{prf} Since $\pi(G)V$ spans a dense subspace of $\cH$, i.e., 
$\cH = \cH_V$ (Proposition~\ref{prop:6.2}), 
the assertion follows from Remark~\ref{rem:2.14}. 
\end{prf}

\begin{rem}
Since $\cH_V \cong \cH_{V'}$ as $G$-representations 
if and only if $V \cong V'$ as $H$-representations 
(cf.\ Corollary~\ref{cor:2.15}), the description of 
all $G$-representations of  positive energy 
for which the $H$-representation $(\rho, V)$ is bounded  
is equivalent to the determination of all bounded $H$-representations 
$(\rho, V)$ for which $(\rho, \beta, V)$ is 
inducible if we put $\beta(\fp^+) = \{0\}$. 
\end{rem}

\begin{cor} \mlabel{cor:6.2} If 
$d \in \g$ is elliptic with $0$ isolated in $\Spec(\ad d)$, 
then every bounded representation of $G$  is holomorphically induced from 
the representation $(\rho,\beta)$ of $(H,\fq)$ on 
$V := \oline{(\cH^\infty)^{\fp^-}}$ defined by $\beta(\fp^+) = \{0\}$. 
\end{cor} 

From Corollary~\ref{cor:commutant} we obtain in particular:

\begin{cor} A positive energy representation 
$(\pi, \cH)$ of $G$ for which the representation $(\rho,V)$ of $H$ 
is bounded is a direct sum of irreducible ones if and only 
if $(\rho, V)$ has this property. 
\end{cor}

\begin{ex}  The complex Banach--Lie algebra 
$\g$ is called {\it weakly root graded} if there exists 
a finite reduced root system $\Delta$ such that $\g$ 
contains the corresponding 
finite dimensional semisimple Lie algebra $\g_\Delta$ and for some 
Cartan subalgebra $\fh \subeq \g_\Delta$, the Lie algebra $\g$ is a direct 
sum of finitely many $\ad \fh$-eigenspaces. 

Now suppose that $\g$ is a real Banach--Lie algebra for which 
$\g_\C$ is weakly root graded, that 
$\fh$ is invariant under conjugation and, for every 
$x \in \fh \cap i\g$, the derivation $\ad x$ has real spectrum.
Then the realization results for bounded unitary representations 
of $G$ which follows from 
\cite[Thm.~5.1]{MNS09}, applied to their holomorphic 
extensions $G_\C \to \GL(\cH)$, can be derived 
easily from Corollary~\ref{cor:6.2}. 
\end{ex}

The following theorem shows that, assuming that $(\pi, \cH)$ is semibounded 
with $d \in W_\pi$ permits us to get rid of the quite implicit assumption 
that the $H$-representation on $V$ is bounded. It is an important 
generalization of Corollary~\ref{cor:6.2} to semibounded representations. 

\begin{thm} \mlabel{thm:6.2b} Let $(\pi, \cH)$ be a semibounded 
unitary representation of the Banach--Lie group $G$ and 
$d \in W_\pi$ be an elliptic element for which $0$ 
is isolated in $\Spec(\ad d)$. 
We write $P \: \fB(\R) \to B(\cH)$ for the spectral measure of the 
unitary one-parameter group $\pi_d(t) := \pi(\exp(td)$. 
Then the following assertions hold: 
\begin{description}
\item[\rm(i)] The representation $\pi\res_H$ of $H$ is semibounded and, 
for each bounded measurable subset $B \subeq \R$, the 
$H$-representation on $P(B)\cH$ is bounded. 
\item[\rm(ii)] The representation $(\pi, \cH)$ is a direct sum of 
representations $(\pi_j, \cH_j)$ for which there exist $H$-invariant 
subspaces $\cD_j \subeq (\cH_j^\infty)^{\fp^-}$ for 
which the $H$-representation $\rho_j$ on $V_j := \oline{\cD_j}$ is bounded and 
$\Spann\big(\pi_j(G)V_j\big)$ is dense in $\cH_j$. 
Then the representations $(\pi_j, \cH_j)$ are holomorphically 
induced from $(\rho_j, \beta_j,V)$, where 
$\beta_j(\fp^+)= \{0\}$. 
\item[\rm(iii)] If $(\pi, \cH)$ is irreducible and 
$s := \inf\Spec(-i\dd\pi(d))$, then $P(\{s\})\cH 
= \oline{(\cH^\infty)^{\fp^-}}$ and 
$(\pi, \cH)$ is holomorphically induced 
from the bounded $H$-representation $\rho$ on this space, extended 
by $\beta(\fp^+) = \{0\}$. 
\end{description}
\end{thm}

\begin{prf} (i) From  $s_{\pi\res_H} = s_\pi\res_{\fh}$ it follows that 
$\pi\res_H$ is semibounded with $d \in W_\pi \cap \fh \subeq W_{\pi\res_H}$. 
For every bounded measurable subset $B \subeq \R$, the relation  
$d \in \z(\fh)$ entails that $P(B)\cH$ is $H$-invariant. Let 
$\rho_B$ denote the corresponding representation of $H$. 
Then the boundedness of $\dd\rho_B(d)$ which commutes with 
$\dd\rho(\fh)$ implies that 
$W_{\rho_B} + \R d = W_{\rho_B}$, so that $d \in W_{\pi\res_H} 
\subeq W_{\rho_B}$ leads to $0 \in W_{\rho_B}$. 
This means that $\rho_B$ is bounded. 

(ii) We apply Zorn's Lemma to the ordered set of all 
pairwise orthogonal systems of closed $G$-invariant subspaces 
satisfying the required conditions. 
Therefore it suffices to 
show that if $\cH \not=\{0\}$, then there exists a non-zero 
$H$-invariant subspaces $\cD \subeq (\cH^\infty)^{\fp^-}$ for 
which the $H$-representation $\rho$ on $\oline{\cD}$ is bounded. 

For $s := \inf\Spec(-i\dd\pi(d))$ and $0 < \eps < \delta$, we 
may take the space $\cD := \break \cH^\infty \cap P([s,s+\eps[)\cH$ 
from the proof of Proposition~\ref{prop:6.2}. 
As we have seen there, it is annihilated 
by $\dd\pi(\fp^-)$ and the boundedness of the $H$-representation on its 
closure follows from~(i). 

That the representations $(\pi_j, \cH_j)$ are holomorphically induced
 from the bounded $H$-representations on $V_j$ follows from 
Theorem~\ref{thm:a.3}. 

(iii) For $t > s$, let 
\[ \cD_U 
:= (\cH^\infty)^{\fp^-} \cap P([s,t[)(\cH^\infty)^{\fp^-}
= (\cH^\infty)^{\fp^-} \cap P(]s-\delta, t[)(\cH^\infty)^{\fp^-}. \] 

{\bf Claim 1:} $\cD_U$ is dense in 
$U := \oline{P([s,t[)(\cH^\infty)^{\fp^-}}$. 

Let $v \in (\cH^\infty)^{\fp^-}$ and $w := P(]s-\delta,t[)v$. 
With Proposition~\ref{prop:c.3}(iii), we find a sequence $f_n \in L^1(\R)$ 
for which $\pi_d(f_n)$ converges strongly to $P(]s-\delta,t[)$ 
and $\supp(\hat f_n) \subeq ]s-\delta, t[$, so that $\pi_d(f_n)v  \to w$ 
and 
\[ \pi_d(f_n)v = P(\hat f_n)v 
= P(]s-\delta, t[)P(\hat f_n)v 
\in P(]s-\delta, t[)\cH^\infty.\] 
Since $\fp^-$ is invariant under $e^{\ad d}$, the closed subspace 
$(\cH^\infty)^{\fp^-}$ is invariant under $\pi_d(\R)$, 
so that $\pi_d(f_n)v\in (\cH^\infty)^{\fp^-}$. 
This proves Claim $1$. 

{\bf Claim 2:} $(\pi, \cH)$ is holomorphically induced from 
the bounded $H$-representation $(\rho, \beta, U)$, defined by 
$\beta(\fp^+) := \{0\}$. 

From (i) we know that the $H$-representation on $U$ is bounded 
and on the dense subspace $\cD_U$ we have $\dd\pi(\fp^-)\cD_U = \{0\}$. 
Therefore (A1/2) in Theorem~\ref{thm:a.3} are satisfied and this proves 
Claim~$2$. 

{\bf Claim 3:} $U = P(\{s\})\cH$ for every $t > s$. 

In view of Claim $2$, Corollary~\ref{cor:commutant} implies that 
the $H$-representation on $U$ is irreducible. Since $\pi_d$ commutes with 
$H$, it follows in particular that $\rho(\exp\R d) \subeq \T\1$. 
The definition of $s$ now shows that $U \subeq P(\{s\})\cH$. 

For $0 < \eps < \delta$ and $t < \delta + \eps$, the proof 
of Proposition~\ref{prop:6.2} implies that 
$\big(P([s,t[)\cH\big) \cap \cH^\infty$ is dense in 
$P([s,t[)\cH$ and contained in $(\cH^\infty)^{\fp^-}$, hence 
in $P([s,t[)(\cH^\infty)^{\fp^-}\subeq U$. 
We conclude in particular that $P(\{s\})\cH \subeq U$. 

{\bf Claim 4:} $U = \oline{(\cH^\infty)^{\fp^-}}$. 

From the definition of $U$ it is clear that 
$U \subeq \oline{(\cH^\infty)^{\fp^-}}$. 
To see that we actually have equality, we note that Claim $2$ 
shows that $P([s,t[)(\cH^\infty)^{\fp^-} \subeq P(\{s\})\cH = U$ 
holds for every $t > s$. 
As $P([s,n]) \to P([s,\infty[) = \id$ holds pointwise, we obtain 
$(\cH^\infty)^{\fp^-} \subeq U$. 

This completes the proof of (iii). 
\end{prf}

\begin{rem} For finite dimensional Lie groups the classification 
of irreducible 
semibounded unitary representations easily boils down to a situation 
where one can apply Theorem~\ref{thm:6.2b}. Here 
$d \in \g$ is a regular element whose centralizer 
$\fh = \ft$ is a compactly embedded Cartan subalgebra and 
the corresponding group $T = H$ is abelian and $V$ is one-dimensional 
(cf.\ \cite{Ne00}). In this case $0$ is trivially isolated in 
the finite set $\Spec(\ad d)$. 
\end{rem}

The following theorem provides a bridge between the seemingly 
weak positive energy condition and the much stronger semiboundedness 
condition. 

\begin{thm} \mlabel{thm:3.15} 
Let $d \in \g$ be elliptic with $0$ isolated in $\Spec(\ad d)$. 
Then a smooth unitary representation 
$(\pi, \cH)$ of $G$ for which the representation $\rho$ of 
$H$ on $\oline{(\cH^\infty)^{\fp^-}}$ is bounded 
satisfies the positive energy condition 
\begin{equation}
  \label{eq:posen}
\inf\Spec(-i\dd\pi(d)) > -\infty
\end{equation}
if and only if $\pi$ is semibounded with $d \in W_\pi$. 
\end{thm}

\begin{prf} If $\pi$ semibounded with $d \in W_\pi$, then 
we have in particular \eqref{eq:posen}. It remains to show the 
converse if all the assumptions of the theorem are satisfied. 
Recall that the splitting condition \eqref{eq:splitcond} 
is satisfied because $0$ is isolated in $\Spec(\ad d)$. 
We note that the representation $\ad_{\fp^+}$ of 
$\fh$ on $\fp^+$ is bounded with $\Spec(\ad_{\fp^+}(-id)) \subeq ]0, \infty[$. 
Therefore the invariant cone 
\begin{equation}
  \label{eq:cone} 
C := \{ x \in \fh \: 
\Spec(\ad_{\fp^+}(-ix)) \subeq ]0,\infty[ \} 
\end{equation}
is non-empty 
and open because it is the inverse image of the 
open convex cone 
\[ \{ X \in \Herm(\fp^+) \: \Spec(X) \subeq ]0, \infty[ \}\] 
(cf.\ \cite[Thm.~14.31]{Up85}) under the continuous linear map 
$\fh \to \gl(\fp^+), x \mapsto \ad_{\fp^+}(-ix)$. 

Let $x \in C$ and 
\[ s_\rho(x) := \sup(\Spec(i\dd\rho(x))
= -\inf(\Spec(-i\dd\rho(x)),\]  so that 
$V := \oline{(\cH^\infty)^{\fp^-}} \subeq \cH^\infty$ 
(Theorem~\ref{thm:6.2}) is contained in 
the spectral subspace $\cH^\infty([-s_\rho(x),\infty[)$ with respect 
to the one-parameter group $t \mapsto \pi(\exp tx)$. 
Since the map 
\[ \g_\C \times \cH^\infty \to \cH^\infty, \quad 
(x,v) \mapsto \dd\pi(x)v \] 
is continuous bilinear and $H$-equivariant, we see with 
Proposition~\ref{prop:spec-add} that 
the subspace $\cH^\infty([-s_\rho(x),\infty[)$ of $\cH^\infty$ 
is invariant under $\fp^+$. 

For every $v \in V \subeq \cH^\omega$, the Poincar\'e--Birkhoff--Witt 
Theorem shows that it 
contains the subspace 
\[ U(\g_\C)v = U(\fp^+)U(\fh_\C)U(\fp^-)v 
= U(\fp^+)U(\fh_\C)v \subeq U(\fp^+)V.\] 
From $V \subeq \cH^\omega$ and $\cH = \cH_V$ it follows that 
$U(\g_\C)V$ is dense in $\cH$, and hence that 
$\cH^\infty([-s_\rho(x),\infty[) \subeq \cH([-s_\rho(x),\infty[)$ 
is dense in $\cH$. We conclude that 
\begin{equation}
  \label{eq:spirho} 
s_{\pi}(x) = \sup(\Spec(i\dd\pi(x))) = s_\rho(x) 
\quad \mbox{ for } \quad x \in C.
\end{equation}

To see that $C \subeq W_\pi$, it now suffices to show that 
$\Ad(G)C$ has interior points. 
Let $\fp := (\fp^+ + \fp^-)\cap \g$ and note that this is a closed 
$H$-invariant complement of $\fh$ in $\g$. 
The map $F \: \fh\times \fp \to \g, F(x,y) := e^{\ad y}x$ 
is smooth and 
\[ \dd F(x,0)(v,w) = [w,x] +  v.\] 
Since the operators $\ad x$, $x \in \h$, 
preserve $\fp$, the operator 
$\dd F(x,0)$ is invertible if and only if 
$\ad x \: \fp \to \fp$ is invertible, and this is the case 
for any $x \in C$ because $\ad x\res_{\fp^\pm}$ are invertible 
operators. This proves that $C$ is contained in the interior 
of $F(C,\fp)$, and hence that $C \subeq W_{\pi}$. 
Therefore $\pi$ is semibounded with $d \in W_{\pi}$. 
\end{prf}

\appendix 

\section{Equicontinuous representations} \mlabel{app:1}

In this appendix we first explain how a continuous representation 
$\pi \: G \to \GL(V)$, i.e., a representation defining a continuous 
action of $G$ on $V$, of the locally compact group $G$ 
on the complete locally convex space $V$ can be integrated to a 
representation of the group algebra $L^1(G)$, provided 
$\pi(G)$ is equicontinuous. If $G$ is abelian, we use this to 
extend Arveson's concept of spectral subspaces to 
representations on complete locally convex spaces. 

\subsection{Equicontinuous groups} 

If $(\pi, V)$ is a continuous representation of a compact group on a 
locally convex space, then $\pi(G)$ is an equicontinuous subgroup of 
$\GL(V)$ (cf.\ Lemma~\ref{lem:1.3}) and if $V$ is finite dimensional, 
then each equicontinuous subgroup of $\GL(V)$ has compact closure. 
Therefore, in the context of locally convex spaces, equicontinuous 
groups are natural analogs of compact groups. 

\begin{proposition} \mlabel{prop:equicont} Let $V$ be a locally convex space. 
For a subgroup $G\subeq \GL(V)$, the following are equivalent: 
\begin{description}
\item[\rm(1)] $G$ is equicontinuous. 
\item[\rm(2)] For each $0$-neighborhood $U \subeq V$, the set 
$\bigcap_{g \in G} gU$ is a $0$-neighborhood. 
\item[\rm(3)] There exists a basis of $G$-invariant absolutely 
convex $0$-neighborhoods in~$V$. 
\item[\rm(4)] The topology on $V$ is defined by the set of $G$-invariant 
continuous seminorms. 
\item[\rm(5)] Each equicontinuous subset of the dual space $V'$ 
is contained in a \break $G$-invariant equicontinuous subset. 
\end{description}
\end{proposition}

\begin{proof} (1) $\Leftrightarrow$ (2): 
The equicontinuity of $G$ means that for each $0$-neighbor\-hood 
$U$ there exists a $0$-neighborhood $W$ in $V$ with 
$gW \subeq U$ for each $g \in G$, i.e., 
$W \subeq \bigcap_{g \in G} gU$. 

\nin (2) $\Leftrightarrow$ (3) follows from the local convexity of $V$. 

\nin (3) $\Leftrightarrow$ (4): For 
each $G$-invariant absolutely convex $0$-neighborhood, the 
corresponding gauge functional is a $G$-invariant continuous seminorm 
and vice versa (\cite[Thm.~1.35]{Ru73}). 

\nin (3) $\Leftrightarrow$ (5): First we note that a subset 
$K \subeq V'$ is equicontinuous if and only if its polar 
$$\hat K := \{ v \in V \: \sup |\la K, v \ra| \leq 1 \} $$
is a $0$-neighborhood and, conversely, for each $0$-neighborhood 
$U \subeq V$, its polar set 
$$\hat U := \{ \lambda \in V' \: \sup |\lambda(U)| \leq 1 \} $$
is equicontinuous. 
If (3) holds and $K \subeq V'$ is equicontinuous, then $\hat K$ 
is a $0$-neighborhood and (3) implies the existence of a $G$-invariant 
absolutely convex $0$-neighborhood $U$, contained in $\hat K$. 
Then $\hat U \supeq \hats K \supeq K$ and the $G$-invariance of $\hat U$ 
imply (5). If, conversely, (5) holds and $W \subeq V$ is a $0$-neighborhood, 
then $\hat W$ is equicontinuous and (5) implies the existence of a 
$G$-invariant equicontinuous subset $K \supeq \hat W$. Then 
$\hat K \subeq W$ is an absolutely convex $G$-invariant $0$-neighborhood 
in $V$. 
\end{proof}

\begin{remark} (a) If $V$ is a normed space, then a subgroup $G \subeq \GL(V)$ 
is equicontinuous if and only if it is bounded. In fact, if $B_r(0)$ is the 
open ball or radius $r$, then the relation 
$G B_r(0) \subeq B_s(0)$ is equivalent to $\|g\| \leq \frac{s}{r}$ 
for all $g \in G$. 
In view of the preceding proposition, this is equivalent to the 
existence of a $G$-invariant norm defining the topology. 

(b) If $G \subeq \GL(V)$ is equicontinuous, then each $G$-invariant 
continuous seminorm $p$ on $V$ defines a homomorphism 
$G \to \Isom(V_p)$, where $V_p$ is the completion of the quotient space 
$V/p^{-1}(0)$ with respect to the norm induced by $p$. We thus obtain an 
embedding 
$G \into \prod_{i \in I} \Isom(V_{p_i}),$
where $(p_i)_{i \in I}$ is a fundamental system of $G$-invariant continuous 
seminorms on $V$. 

(c) If $G \subeq \GL(V)$ is equicontinuous, then 
Proposition~\ref{prop:equicont}(5) implies that each $G$-orbit 
in $V'$ is equicontinuous. 
\end{remark}

\begin{lemma} \mlabel{lem:1.3} 
If $\alpha \: G \to \GL(V)$ defines a continuous action of 
the compact group $G$ on the locally convex space $V$, then 
$\alpha(G)$ is equicontinuous. 
\end{lemma}

\begin{proof} We write $\sigma \: G \times V \to V, (g,v) \mapsto \alpha(g)v$ for the 
action of $G$ on $V$. Let $U \subeq V$ be a convex balanced $0$-neighborhood. 
Then 
$U_G := \bigcap_{g \in G} \alpha(g)(U)$
also is a zero neighborhood in $V$ because 
$\sigma^{-1}(U) \subeq G \times V$ is an open neighborhood of the 
compact subset $G \times \{0\}$. Hence there exists some open 
$0$-neighborhood $W \subeq V$ with $G \times W \subeq \sigma^{-1}(U)$, 
and this means that $W \subeq U_G$. Now we apply 
Proposition~\ref{prop:equicont}.
\end{proof}

\subsection{Arveson spectral theory on locally convex spaces} 

\begin{defn} A representation $(\pi, V)$ of the group 
$G$ on the locally convex space $V$ is called {\it equicontinuous} 
if $\pi(G) \subeq B(V)$ (the space of continuous linear operators on $V$) 
is an equicontinuous group of operators, 
which is equivalent to the existence 
of a family of $G$-invariant continuous seminorms 
defining the topology (cf.\ Proposition~\ref{prop:equicont}). 
\end{defn}

\begin{defn} \mlabel{def:arv} (a) Let $(\pi, G)$ be an equicontinuous 
strongly continuous action of the locally compact group $G$ 
on the complete complex locally convex space $V$. 

We write $\cP^G$ for the set of $\pi(G)$-invariant continuous seminorms 
on $V$. This set carries a natural order defined by 
$p\leq q$ if $p(v) \leq q(v)$ for every $v \in V$. 
Let $V_p$ be the completion of $V/p^{-1}(0)$ with respect to~$p$. 
For $p \leq q$ we then obtain a contractive map 
$\phi_{pq} \: V_q \to V_p$ of Banach spaces. 
Since $V$ is complete, the natural map 
\[ V \to \prolim V_p 
:= \Big\{ 
(v_p) \in \prod_{p \in \cP^G} V_p\: (\forall p,q \in \cP^G)\, p \leq q \Rarrow 
\phi_{pq}(v_q) = v_p\Big\} 
\subeq \prod_{p \in \cP^G} V_p \] 
is a topological isomorphism, where the right hand side carries 
the product topology. 

We now obtain on each $V_p$ a continuous isometric representation 
$(\pi_p, V_p)$ of $G$. Therefore each $f \in L^1(G)$ defines a bounded 
operator 
\[ \pi_p(f) := \int_G f(g)\pi_p(g)\, dg \in B(V_p) \] 
where $dg$ stands for a left Haar measure on $G$ 
(\cite[(40.26)]{HR70}). Since these operators 
satisfy for $p \leq q$ the compatibility relation 
$\phi_{pq} \circ \pi_q(f) = \pi_p(f)$, 
the product operator 
$(\pi_p(f))_{p \in\cP^G}$ on $\prod_{p \in \cP^G} V_p$ preserves the closed 
subspace $\prolim V_p \cong V$, 
so that we obtain a continuous linear operator $\pi(f) \in 
B(V)$. This leads to a representation 
\begin{equation}
  \label{eq:intrep}
\pi \: (L^1(G), *) \to B(V).
\end{equation}

(b) If $G$ is abelian and $\hat G := \Hom(G,\T)$ is its character group, 
then we have the Fourier transform 
\[ \cF \: L^1(G) \to C_0(\hat G), \quad 
\cF(f)(\chi) := \int_G \chi(g) f(g)\, dg.\] 
We define the {\it spectrum of $(\pi, V)$} by 
\[ \Spec(V) := \Spec_\pi(V) := \{ \chi \in \hat G \: 
(\forall f \in L^1(G))\, \pi(f) = 0 \Rarrow \hat f(\chi) = 0\},\]
i.e., the {\it hull} or {\it cospectrum} of the ideal 
$\ker \pi \trile L^1(G)$. 
Accordingly, we define the {\it spectrum of an element $v \in V$} 
by 
\[ \Spec(v) := \Spec_\pi(v) := \{ \chi \in \hat G \: 
(\forall f \in L^1(G))\, \pi(f)v = 0 \Rarrow \hat f(\chi) = 0\},\]
which is the hull of the annihilator ideal of $v$. 
For a subset $E \subeq \hat G$, we now define 
the corresponding {\it Arveson spectral subspace} 
\[ V(E) := V(E,\pi) := \{ v \in V\: \Spec(v) \subeq \oline E \}\] 
(cf.~\cite[Sect.~2]{Ar74}, where this space is denoted $M^\pi(\oline E)$; 
see also \cite[Ch.~XI]{Ta03}). 
\end{defn}

\begin{rem} \mlabel{rem:a.6}
In \cite[p.~225]{Ar74} it is shown that 
\[ V(E) = \{ v \in V \: (\forall f \in L^1(G))\ 
\supp(\hat f) \cap \oline E = \eset \Rarrow  \pi(f)v = 0\}, \] 
which implies in particular that $V(E)$ is a closed subspace which is 
clearly $\pi(G)$-invariant. 
Note that the condition $\supp(\hat f) \cap 
\oline E = \eset$ 
means that $\hat f$ vanishes on a neighborhood of~$\oline E$. 
\end{rem}

\begin{rem} \mlabel{rem:joe}
In \cite[Thm.~3.3]{Jo82} it is shown that for one-parameter groups 
of isometries of a Banach space $X$ with infinitesimal generator $A$, 
the Arveson spectrum of an 
element $v \in X$ coincides with the complement of the maximal open 
subset $G \subeq \R$ for which the map $z \mapsto (z - A)^{-1}v$ 
extends holomorphically to $(\C\setminus \R) \cup G$ 
(see also \cite[Ch.~XI]{Ta03}). 
\end{rem}

%\item Check \cite{Ta03} for the equality of Arveson spectrum and 
%spectrum of the infinitesimal generator 
%To simplify the arguments, we assune from now on that $G = \R^n$ 
%and that $(\pi, V)$ is a continuous equicontinuous representation 
%of $G$ on a complete locally convex space. 

\begin{rem} \mlabel{rem:intersect} 
If $(E_j)_{j \in J}$ is a family of closed subsets of $\hat G$, 
then 
\[ V(\bigcap_{j\in J} E_j) = \bigcap_{j \in J} V(E_j)\] 
follows immediately from the definition. 
\end{rem}

\begin{lem} \label{lem:a.10} {\rm(\cite[p.~226]{Ar74})} For 
$\chi \in \hat G$ we have
\[ V(\{\chi\}) = V_\chi(\pi) := 
\{ v \in V \: (\forall g\in G)\, \pi(g)v = \chi(g)v\}.\] 
\end{lem}

\begin{defn} Let $\g$ be a complete locally convex Lie algebra 
and $x \in \g$ be such that $\ad x$ generates a continuous 
equicontinuous one-parameter 
group $\pi \: \R \to \Aut(\g)$ of automorphisms, i.e., 
$\pi$ is strongly differentiable with $\pi'(0) = \ad x$. 
Then we obtain for each closed subset $E \subeq \R$ a spectral 
subspace $\g_\C(E)$. 
\end{defn}

\begin{defn} For a subset $E \subeq \hat G$,  we also define 
\[ R(E):= R^\pi(E):= \oline{\Spann\{ \pi(f)V \: f \in L^1(G), 
\hat f \in C_c(\hat G), 
\supp(\hat f) \subeq E\}}.\]  
\end{defn}

\begin{rem} \label{rem:a.13} 
(a) Clearly $E_1 \subeq E_2$ implies $R(E_1) \subeq R(E_2)$ 
and Remark~\ref{rem:a.6} implies that $R(E) \subeq V(E)$.
Here we use that, for $f,h \in L^1(G)$, the conditions 
$\supp(\hat f) \subeq E$ and 
$\supp(\hat h) \cap \oline E =\eset$ imply 
that $(h * f)\,\hat{} = \hat h \hat f =0$, and this leads to 
$h * f = 0$, which in turn shows that 
$\pi(h)\pi(f)v  = \pi(hf)v = 0$. 
\end{rem}

As we shall see below, the following lemma is an important 
technical tool. 

\begin{lem} \mlabel{lem:intersect} 
For every closed subset $E \subeq \hat G$, the space 
$V(E)$ is the intersection of the subspaces 
$R(E + N)$, where $N$ is an identity neighborhood in $\hat G$. 
\end{lem}

\begin{prf} Since every identity neighborhood $N$ in $\hat G$ 
contains a compact one, it suffices to consider the spaces 
$R(E+N)$ for identity neighborhoods $N$ which are compact. 
Hence the assertion follows from \cite[Prop.~2.2]{Ar74} if 
$V$ is a Banach space. 

We first show that 
$V(E) \subeq R(E + N)$. We know already that this is the case 
if $V$ is Banach. To verify it in the general case, 
let $v \in V(E)$ and $U$ be a neighborhood of $v$. 
We have to show that $U$ intersects $R(E + N)$. 
Let $p \in \cP^G$ be a $G$-invariant seminorm for which 
the $\eps$-ball $U^p_\eps(v) = \{ w \: p(v-w) < \eps\}$ is contained in $U$. 
Since the result holds for the $G$-representation 
$(\pi_p, V_p)$ on the Banach space $V_p$ and the image 
$v_p$ of $v$ in this space belongs to $V_p(E)$, the  open 
$\eps$-ball $U_\eps(v_p)$ in $V_p$ intersects $R(E + N, \pi_p)$, 
hence contains a finite sum 
$\sum_{j = 1}^n \pi_p(f_j) w_j$ with 
$\supp(\hat f_j) \subeq E + N$ compact. 
Since the image of $V$ in $V_p$ is dense and $U_\eps(v_p)$ is open, 
we may even assume that 
$w_j = v_{j,p}$ for some $v_j \in V$. 
Then 
\[ p\Big(v - \sum_{j = 1}^n \pi(f_j) v_j\Big)
= p\Big(v_p - \sum_{j = 1}^n \pi_p(f_j) v_{j,p}\Big) < \eps \] 
implies that $\sum_{j = 1}^n \pi(f_j) v_j \in U \cap R(E + N)$. 
We conclude that $v \in \oline{R(E + N)} = R(E+N)$. 

We now arrive with Remarks~\ref{rem:intersect} and \ref{rem:a.13} at 
\[ V(E) \subeq \bigcap_N R(E+ N) \subeq 
\bigcap_N V(E + N) = V\Big( \bigcap_N (E + N)\Big) = V(E), \] 
because all sets $E + N$ are closed. 
\end{prf}

\begin{prop}\mlabel{prop:spec-add} 
Assume that $(\pi_j, V_j)$, $j =1,2,3$ are 
continuous equicontinuous representations of the locally 
compact abelian group $G$ on the complete locally convex spaces $V_j$ and that 
$\beta \: V_1 \times V_2 \to V_3$ 
is a continuous equivariant bilinear map. 
Then we have for closed subsets $E_1, E_2 \subeq \hat G$ the relation 
\[ \beta(V_1(E_1) \times V_2(E_2)) 
\subeq V_3(E_1 + E_2). \] 
\end{prop}

\begin{prf} 
Pick an identity neighborhood $U_0 \subeq \hat G$ and choose a 
symmetric $\1$-neighborhood $U$ with $UU \subeq U_0$. 
In view of Lemma~\ref{lem:intersect}, 
the assertion follows, if we can show that 
\[ \beta(R^{\pi_1}(E_1+U) \times R^{\pi_2}(E_2+U))  
\subeq V_3(E_1 + E_2 + U_0). \]  
To verify this relation, pick $v_1, v_2 \in V$ 
and $f_1, f_2 \in L^1(\R)$ such that 
the functions $\hat f_j$ are compactly supported with 
$\supp(\hat f_j) \subeq E_j + U$. We have to show that 
\[ \beta(\pi_1(f_1) v_1, \pi_2(f_2)v_2) \in V_3(E_1 + E_2 + U_0), \]  
i.e., that for any $f_3 \in L^1(\R)$ for which $\hat f_3$ vanishes on 
a neighborhood of $\oline{E_1 + E_2 + U_0}$, we have 
\[ \pi_3(f_3)\beta(\pi_1(f_1)v_1, \pi_2(f_2)v_2) = 0. \] 
This expression can easily be evaluated to 
\begin{align*}
&\pi_3(f_3)\beta(\pi_1(f_1)v_1, \pi_2(f_2)v_2) \\ 
&= \int_G \int_G \int_G f_3(g_3) f_1(g_1)f_2(g_2) \pi_3(g_3)\beta(\pi_1(g_1)v_1, 
\pi_2(g_2)v_2)\, dg_1 d{g_2} d{g_3} \\ 
&= \int_G \int_G \int_G 
f_3(g_3) f_1(g_1)f_2(g_2) \beta(\pi_1(g_3+g_1)v_1, 
\pi_2(g_3+g_2)v_2)\, dg_1 d{g_2} d{g_3} \\ 
&= \int_G \int_G \int_G 
f_3(g_3) f_1(g_1-g_3)f_2(g_2-g_3) \beta(\pi_1(g_1)v_1, 
\pi_2(g_2)v_2)\, dg_1 d{g_2} d{g_3} \\ 
&= \int_G \int_G \int_G 
f_3(g_3) f_1(g_1-g_3)f_2(g_2-g_3) \beta(\pi_1(g_1)v_1, 
\pi_2(g_2)v_2)\,  d{g_3}dg_1 d{g_2} \\ 
&= \int_G \int_G (f_3 * F)(g_1, g_2)\beta(\pi_1(g_1)v_1, 
\pi_2(g_2)v_2)\,  d{g_3}dg_1 d{g_2} \\ 
\end{align*}
for $F(g_1, g_2) := f_1(g_1) f_2(g_2)$ and 
\[ (f_3 * F)(g_1, g_2) = \int_G f_3(g) f_1(g_1-g)f_2(g_2-g)\, dg.\] 
The Fourier transform of $f_3 * F \in L^1(G \times G)$ is given by 
\begin{align*}
(f_3 * F)\,\hat{}(\chi_1, \chi_2) 
&= \int_G \int_G \int_G f_3(g) f_1(g_1-g)f_2(g_2-g)\chi_1(g_1)\chi_2(g_2)
\, dg\, dg_1\, dg_2\\
&= \int_G f_3(g) \chi_1(g)\chi_2(g) \int_G \int_G f_1(g_1)f_2(g_2)
\chi_1(g_1)\chi_2(g_2)\, dg_1\, dg_2\, dg\\
&= \hat f_1(\chi_1) \hat f_2(\chi_2) \hat f_3(\chi_1+\chi_2).
\end{align*}
By assumption, $\hat f_3$ vanishes on 
\[ \supp(\hat f_1)+\supp(\hat f_2)
\subeq E_1 + E_2 + U + U \subeq E_1 + E_2 + U_0.\] 
Therefore $(f_3 * F)\,\hat{}=0$, and hence also $f_3 * F =0$. 
This completes the proof. 
\end{prf}

\begin{rem} \mlabel{rem:a.15} Now we consider the case $G= \R$ and a 
strongly continuous representation $(\alpha, V)$ of $\R$ by 
isometries. We know already from Lemma~\ref{lem:a.10} that 
$V_0 := V(\{0\})$ is the set of fixed points. We would like 
to have a decomposition 
\begin{equation}
  \label{eq:trideco}
V = V_+ \oplus V_0 \oplus V_-,
\end{equation}
where $V_\pm$ are closed invariant subspaces satisfying 
\[ V_+ \subeq V([\delta, \infty[) \quad \mbox{ and } \quad 
V_- \subeq V(]-\infty, -\delta]) \] 
for some $\delta > 0$. 

We claim that this implies that the latter inclusions are equalities. 
In fact, for any three pairwise disjoint closed subsets 
$E_1, E_2, E_3$, it follows from 
the relation 
\[ V(E_1) + V(E_2) \subeq V(E_1 \cup E_2)\] 
and 
Remark~\ref{rem:intersect} that the sum 
$V(E_1) + V(E_2) + V(E_3)$ is direct. Our assumption 
\eqref{eq:trideco} now implies 
that $V$ satisfies 
\begin{equation}
  \label{eq:splitco} 
 V(]-\infty, -\delta])  \oplus  V(\{0\}) \oplus  V([\delta, \infty[) = V.\tag{SC}
\end{equation}
Since this sum is direct, it follows that 
\begin{equation}
V_+ = V([\delta, \infty[) \quad \mbox{ and } \quad 
V_- = V(]-\infty, -\delta]). 
\end{equation}
Condition (SC) is called the {\it splitting condition}. 
\end{rem} 

\begin{lem}
  \label{lem:a.17}
If $D := \alpha'(0)$ is a bounded operator, i.e., 
$\alpha \: \R \to \Aut(V)$ is norm continuous, then the splitting 
condition \eqref{eq:splitco}  is satisfied if and only if 
$0$ is isolated in the spectrum of the hermitian operator 
$D$. 
\end{lem} 

\begin{prf} First we recall from Remark~\ref{rem:joe} that the 
Arveson spectrum $\Spec(V)$ coincides with the spectrum of 
$D$, which is a subset of $i\R$.

If $0$ is isolated in $\Spec(D)$, then 
$\Spec(iD) \subeq ]-\infty,-\delta] \cup \{ 0\} \cup 
[\delta, \infty[$ 
for some $\delta > 0$, and the splitting condition follows from 
the existence of corresponding spectral projections defined by 
contour integrals (see also \cite[Thm.~1.4]{Bel04}, where spectral subspaces 
are defined in terms of Dunford's local spectral theory.) 

Suppose, conversely, that $D$ is not invertibble and that the 
splitting condition is satisfied. Then 
$\im(D) = V_+ \oplus V_-$ is a closed subspace, $V_0 = \ker D$, and 
$V = \im(D) \oplus \ker(D)$. Hence \cite{Si70} implies that 
$0$ is isolated in $\Spec(D)$. 
\end{prf}

\end{document}